\documentclass[12pt,reqno]{amsart}

\usepackage{a4wide}

\pdfoutput=1

\usepackage[utf8]{inputenc}
\usepackage[english]{babel}
\usepackage{amsmath,amsfonts,amsthm,amssymb,amscd,amsbsy}
\usepackage[all]{xy}
\usepackage{graphicx}
\usepackage{euscript}
\usepackage{mathtext}
\usepackage{upgreek}
\usepackage{hyperref}
\hypersetup{
    colorlinks=true,
    linkcolor=blue,
    filecolor=magenta,      
    urlcolor=cyan,
}
\usepackage{euscript}
\usepackage{mathtools}
\usepackage{microtype}
\usepackage{setspace}
\usepackage{fancyhdr}
\usepackage{pict2e}
\usepackage{mathrsfs}
\usepackage[shortlabels]{enumitem}
\usepackage{stackengine}
\usepackage{changepage}
\usepackage{parskip}
\usepackage{tikz}
\usetikzlibrary{arrows}
\usepackage{lmodern}
\usepackage{csquotes}

\usepackage{xpatch}
\xapptocmd\normalsize{%
 \abovedisplayskip=10pt plus 1pt minus 3pt
 \abovedisplayshortskip=1pt plus 3pt
 \belowdisplayskip=10pt plus 2pt minus 3pt
 \belowdisplayshortskip=8pt plus 3pt minus 2pt
}{}{}

\numberwithin{equation}{section}

\usepackage[
backend=biber,
style=ieee,
sorting=nyt,
giveninits=true,
dashed=false,
maxbibnames=99
]{biblatex}
\DeclareSymbolFont{largesymbols}{OMX}{cmex}{m}{n}

\newtheoremstyle{theoremstyle}
  {0mm} 
  {0mm} 
  {\itshape} 
  {} 
  {\bfseries} 
  {.} 
  {.5em} 
  {} 

\theoremstyle{theoremstyle}
\newtheorem{proposition}{Proposition}[section]
\newtheorem{lemma}[proposition]{Lemma}
\newtheorem{theorem}[proposition]{Theorem}

\newtheorem{corollary}{Corollary}[proposition]

\newtheoremstyle{examplestyle}
  {0mm} 
  {0mm} 
  {} 
  {} 
  {\bfseries} 
  {.} 
  {.5em} 
  {} 

\theoremstyle{examplestyle}
\newtheorem{remark}[proposition]{Remark}

\newtheoremstyle{remarks}
  {0mm} 
  {0mm} 
  {} 
  {} 
  {\bfseries} 
  {.} 
  {-.5em} 
  {} 

\theoremstyle{remarks}
\newtheorem*{remarks}{Remarks}

\renewcommand{\Re}{\hspace{0.07em}\mathrm{Re}\hspace{0.07em}}
\renewcommand{\Im}{\hspace{0.08em}\mathrm{Im}\hspace{0.04em}}
\newcommand{\ccdot}{\hspace{-1.2pt} \cdot \hspace{-0.5pt}}

\newcommand{\rk}{\mathrm{rk} \hspace{0.1em}}

\newcommand{\tr}{\mathrm{tr}}

\renewcommand{\mod}{\mathrm{mod} \hspace{0.2em}}

\newcommand{\Aut}{\mathrm{Aut} \hspace{0.06em}}

\newcommand{\End}{\mathrm{End}}

\newcommand{\Inn}{\mathrm{Inn}}
\newcommand{\DD}{\mathrm{DD}}
\newcommand{\Lie}{\mathrm{Lie} \hspace{0.06em}}

\newcommand{\defqe}{=\vcentcolon}
\newcommand{\dist}{\mathrm{dist}}
\newcommand{\Rl}{\mathbb{R}}
\newcommand{\Cx}{\mathbb{C}}
\newcommand{\Hq}{\mathbb{H}}
\newcommand{\Oo}{\mathbb{O}}
\newcommand{\Ad}{\mathrm{Ad} \hspace{0.06em}}
\newcommand{\ad}{\mathrm{ad} \hspace{0.06em}}
\newcommand{\mk}[1]{\mathfrak{{#1}}}
\newcommand{\mb}[1]{\mathbb{{#1}}}
\newcommand{\pr}{\mathrm{pr}}
\newcommand{\vspan}{\mathrm{span}}

\newcommand{\set}[1]{\hspace{-0.8pt} \left \{ \hspace{0.03em} {#1} \hspace{0.03em} \right \}}
\newcommand{\bilin}[2]{\langle \hspace{0.1em} {#1} \hspace{0.1em} , \hspace{0.1em} {#2} \hspace{0.1em} \rangle \hspace{0.02em}}
\newcommand{\cross}[2]{\langle \hspace{0.1em} {#1} \hspace{0.1em} | \hspace{0.1em} {#2} \hspace{0.1em} \rangle \hspace{0.02em}}
\newcommand{\restr}[2]{{\left.\kern-\nulldelimiterspace #1 \vphantom{\big|} \right|_{#2}}}

\newcommand{\II}{I \hspace{-0.21em} I}
\newcommand{\Gr}{\mathrm{Gr}}
\newcommand{\Spin}{\mathrm{Spin}}

\newcommand{\GL}{\mathrm{GL}}
\newcommand{\SL}{\mathrm{SL}}
\renewcommand{\O}{\mathrm{O}}
\newcommand{\SO}{\mathrm{SO}}
\newcommand{\U}{\mathrm{U}}
\newcommand{\SU}{\mathrm{SU}}
\newcommand{\Sp}{\mathrm{Sp}}
\newcommand{\Ss}{\mathrm{S}}

\newcommand\altxrightarrow[2][0pt]{\mathrel{\ensurestackMath{\stackengine%
  {\dimexpr#1-7.5pt}{\xrightarrow{\phantom{#2}}}{\scriptstyle\!#2\,}%
  {O}{c}{F}{F}{S}}}}
\newcommand{\isoto}{\altxrightarrow[1pt]{\sim}}

\newcommand{\mysetminusD}{\hbox{\tikz{\useasboundingbox (-0.5pt,-0.5pt) rectangle (5pt,8pt); \draw[line width=0.6pt,line cap=round] (3.5pt,-1.5pt) -- (0,7.25pt);}} \hspace{-1pt}}
\newcommand{\mysetminusT}{\mysetminusD}
\newcommand{\mysetminusS}{\hbox{\tikz{\draw[line width=0.45pt,line cap=round] (2pt,0) -- (-0.5pt,5pt);}} \hspace{1pt}}
\newcommand{\mysetminusSS}{\hbox{\tikz{\draw[line width=0.4pt,line cap=round] (1.5pt,0) -- (0,3pt);}}}

\newcommand{\mysetminus}{\mathbin{\mathchoice{\mysetminusD}{\mysetminusT}{\mysetminusS}{\mysetminusSS}}}
 
\makeatletter
\newcommand{\extp}{\@ifnextchar^\@extp{\@extp^{\,}}}
\def\@extp^#1{\mathop{\bigwedge\nolimits^{\!#1}}}
\makeatother

\newcommand{\overbar}[1]{\mkern 1.5mu\overline{\mkern-2mu#1\mkern-1.5mu}\mkern 1.5mu}

\makeatletter
\DeclareRobustCommand{\loplus}{\mathbin{\mathpalette\dog@lsemi{+}}}
\DeclareRobustCommand{\roplus}{\mathbin{\mathpalette\dog@rsemi{+}}}

\newcommand{\dog@rsemi}[2]{\dog@semi{#1}{#2}{-90,90}}
\newcommand{\dog@lsemi}[2]{\dog@semi{#1}{#2}{270,90}}
\newcommand{\dog@semi}[3]{%
  \begingroup
  \sbox\z@{$\m@th#1#2$}%
  \setlength{\unitlength}{\dimexpr\ht\z@+\dp\z@\relax}%
  \makebox[\wd\z@]{\raisebox{-\dp\z@}{%
    \begin{picture}(1,1)
    \linethickness{\variable@rule{#1}}
    \roundcap
    \put(0.5,0.5){\makebox(0,0){\raisebox{\dp\z@}{$\m@th#1#2$}}}
    \put(0.5,0.5){\arc[#3]{0.5}}
    \end{picture}%
  }}%
  \endgroup
}
\newcommand{\variable@rule}[1]{%
  \fontdimen8  
  \ifx#1\displaystyle\textfont3\else
    \ifx#1\textstyle\textfont3\else
      \ifx#1\scriptstyle\scriptfont3\else
        \scriptscriptfont3\relax
  \fi\fi\fi
}
\makeatother

\begin{document}

\title[Homogeneous hypersurfaces in symmetric spaces of rank two]{Classification of homogeneous hypersurfaces in some noncompact symmetric spaces of rank two}
\author{Ivan Solonenko}
\address{Department of Mathematics, King's College London, United Kingdom}
\email{ivan.solonenko@kcl.ac.uk}

\begin{abstract}
We classify, up to isometric congruence, the homogeneous hypersurfaces in the Riemannian symmetric spaces $\SL(3,\Hq)/\Sp(3), \hspace{1pt} \SO(5,\Cx)/\SO(5),$ and $\mathrm{Gr}^*(2,\Cx^{n+4}) = \SU(n+2,2)/\Ss(\U(n+2)\U(2)), \, n \geqslant 1$. 
\end{abstract}

\maketitle

\vspace{-1em}

\section{Introduction}

A properly embedded connected hypersurface $S$ in a complete connected Riemannian manifold $M$ is called homogeneous if the subgroup of isometries of $M$ preserving $S$ acts transitively on $S$. A proper isometric action of a connected Lie group $G$ on $M$ is said to be of cohomogeneity one if it has a codimension one orbit, and two such actions are called orbit equivalent if there is an isometry of $M$ mapping the orbits of one action onto the orbits of the other. Nonsingular orbits of cohomogeneity-one actions are homogeneous hypersurfaces and, conversely, every homogeneous hypersurface arises in this way. Since a cohomogeneity-one action is uniquely determined by any of its orbits, the problem of classification of homogeneous hypersurfaces up to isometric congruence is equivalent to that of cohomogeneity-one actions up to orbit equivalence. Such actions have been in the limelight because they can be used to construct various geometric structures on manifolds, like Einstein metrics or metrics with special holonomy. In the context of Riemannian symmetric spaces, they also reveal new connections between geometric and algebraic properties of the space as well as provide the chief example of polar and hyperpolar actions.

The classification of cohomogeneity-one actions on irreducible symmetric spaces of compact type was obtained by Kollross (see \cite{kollrossclassification} and further references there). The reducible case is still open. In the noncompact rank-one situation, the classification was obtained by Berndt, Tamaru, and Br\"uck for real and complex hyperbolic spaces and the Cayley hyperbolic plane (\cite{berndt_bruck}, \cite{berndttamarurankone}) and by D\'{ı}az-Ramos, Dom\'{ı}nguez-V\'azquez, and Rodr\'{ı}guez-V\'azquez for quaternionic hyperbolic spaces (\cite{protohomogeneous}). The higher rank noncompact case, however, presents a significantly more difficult problem. It is a standard fact about cohomogeneity-one actions on Hadamard manifolds that any such action has at most one singular orbit (and no exceptional orbits). Berndt and Tamaru classified cohomogeneity-one actions on irreducible symmetric spaces of noncompact type that have either no singular orbits (\cite{berndttamarufoliations}) or a totally geodesic singular orbit (\cite{berndttamarusingtotgeod}). They later went on to prove that any cohomogeneity-one action on such a space $M$ that has a non-totally-geodesic singular orbit must arise from one of two new methods proposed in \cite{berndttamarucohomogeneityone}: the canonical extension and the nilpotent construction. The canonical extension method allows one to obtain cohomogeneity-one actions on $M$ from those on certain totally geodesic submanifolds of $M$ called boundary components. These submanifolds are noncompact symmetric spaces of lower rank, which, in theory, allows one to proceed by induction on rank. The nilpotent construction method reduces the search for cohomogeneity-one actions to a certain problem in representation theory of reductive Lie groups. This method is arguably more involved and has so far produced only two new examples not arising through any other technique. It is noteworthy that both of them are actions on exceptional symmetric spaces related to the exceptional Lie group $G_2$.

Using the above results, Berndt and Tamaru obtained the explicit classification of cohomoge\-neity-one actions on symmetric spaces $\SL(3,\Rl)/\SO(3)$ and $G_2^2/\SO(4)$ (\cite{berndttamarucohomogeneityone}). Later, Berndt and Dom\'{ı}nguez-V\'azquez refined the nilpotent construction method, which allowed them to obtain the classification on $G_2(\Cx)/G_2, \hspace{1pt} \SL(3,\Cx)/\SU(3),$ and $\mathrm{Gr}^*(2, \Rl^{n+4}) = \SO^0(n+2,2)/\SO(n+2)\SO(2), \, n \geqslant 1$ (\cite{berdntdominguez-vazquez}).

The purpose of the present article is to coalesce ideas from these two papers as well as develop some new computational techniques to obtain the explicit classification of cohomogeneity-one actions -- and thus homogeneous hypersurfaces -- on several noncompact symmetric spaces of rank 2 that have not been considered yet, namely on $\SL(3,\Hq)/\Sp(3), \hspace{1pt} \SO(5,\Cx)/\SO(5),$ and $\mathrm{Gr}^*(2,\Cx^{n+4}) = \SU(n+2,2)/\Ss(\U(n+2)\U(2)), n \geqslant 1$. The first of them is the quaternionic analog of the spaces $\SL(3,\Rl)/\SO(3)$ and $\SL(3,\Cx)/\SU(3)$ that have already been dealt with. In order to handle this space, we have to use the classification of so-called protohomogeneous subspaces of $\Hq^n$ that was recently obtained in \cite{protohomogeneous}. The space $\mathrm{Gr}^*(2,\Cx^{n+4})$ is the complex analog of the real noncompact Grassmannian of two-planes studied by Berndt and Dom\'{ı}nguez-V\'azquez. It is also the noncompact dual of the ordinary Grassmannian of complex two-planes, a space whose differential geometry was studied in depth by Berndt in \cite{berndtgrassmannian}. We adapt some results from that paper to deal with the congruence problem for actions on $\mathrm{Gr}^*(2,\Cx^{n+4})$. Finally, $\SO(5,\Cx)/\SO(5)$ is a space with root system of type $B_2$ dual to the compact symmetric space $\Spin(5)$. The nilpotent construction problem is not so onerous for this space, so we handle it using rather elementary methods and no auxiliary results. The main result of the paper is that the nilpotent construction method yields no new actions on these spaces. This provides further evidence to the hypothesis made by Berndt and Tamaru in \cite{berndttamarucohomogeneityone} that the nilpotent construction only yields new actions for exceptional symmetric spaces, and perhaps only those related to $G_2$.

There are still a few noncompact symmetric spaces of rank 2 for which the classification of cohomogeneity-one actions remains open, namely $E_6^{-26}/F_4$ (type $A_2$), $\Sp(2,2)/\Sp(2)\Sp(2)$ (type $C_2$), $\Gr^*(2, \Hq^{n+4}), E_6^{-14}/\Spin(10)\U(1),$ and $\SO(5,\Hq)/\U(5)$ (all three of type $(BC)_2$). Unsurprisingly, the main obstacle to the classification for all of them is the complexity of the nilpotent construction problem. For example, for $E_6^{-26}/F_4$ it involves dealing with the 16-dimensional spin representation of $\Spin(9)$, while for $\Sp(2,2)/\Sp(2)\Sp(2)$ one has a 10-dimensional representation of the Lorentz group $\SO^0(5,1)$. Undoubtedly, novel conceptual methods are needed if one hopes to solve the classification problem for all symmetric spaces of noncompact type.

The paper is organized as follows. Section \ref{preliminaries} reviews the terminology, notation, and previous results needed to understand the rest of the paper. We also introduce some new terminology helpful for studying the nilpotent construction problem (Subsection \ref{nilpotent construction}). Apart from that, we correct a small conceptual error that occurred in \cite{berndttamarucohomogeneityone} and a few other publications of its authors (Subsection \ref{error correction}) and prove a useful result for noncompact rank-$1$ symmetric spaces allowing one to represent cohomogeneity-one actions with a singular orbit on them in a uniform manner (Subsection \ref{rank one case}). Sections \ref{first space}, \ref{second space}, and \ref{third space} are each dedicated to one of the three spaces in question.

\textbf{Acknowledgments.} I am enormously indebted to my supervisor J\"{u}rgen Berndt for his continuous guidance and support, invaluable advice, and an uncanny ability to always give just the right reference.

\section{Preliminaries}\label{preliminaries}

This section serves a number of purposes. First, it is intended to establish the terminology and notation necessary to understand Berndt and Tamaru's classification results as well as the rest of the paper. Second, it provides a slightly more general perspective on the nilpotent construction. Finally, it introduces a few useful results that will aid in describing certain cohomogeneity-one actions later in the paper. We follow the notation established in \cite{berdntdominguez-vazquez} and \cite{berndttamarucohomogeneityone} and refer to the same articles as well as \cite{helgason} and \cite{knapp} for more detailed expositions.

\subsection{Parabolic subalgebras and subgroups.} Given a symmetric space $M$, one usually represents it as a quotient $G/K$, where $G = I^0(M)$ and $K$ is the isotropy subgroup of $G$ at any point. In practice, however, many symmetric spaces are given as quotients $G/K$, where the Riemannian symmetric pair $(G,K)$ is not effective, i.e. the action of $G$ on $G/K$ is not effective, so we will allow some degree of freedom as well.

Let $(G, K)$ be a Riemannian symmetric pair of noncompact type. Let $(\mk{g}, \uptheta)$ denote the corresponding orthogonal symmetric Lie algebra. Recall that we have:

\begin{enumerate}
    \item $G$ and $\mk{g}$ are semisimple.
    \item $\uptheta$ is uniquely determined by $K$, it is a Cartan involution on $\mk{g}$ and gives a Cartan decomposition $\mk{g} = \mk{k} \oplus \mk{p}$ with $\mk{k} = \Lie(K)$.
    \item $\uptheta$ admits a unique lift to $G$, which we denote by $\Uptheta$ and sometimes call a global Cartan involution.
    \item $G^\Uptheta = K$.
    \item $Z(G) \subseteq K$ is a discrete subgroup.
\end{enumerate}

We assume that $\mk{k}$ contains no nontrivial ideals of $\mk{g}$ or, equivalently, that $\mk{g}$ has no compact ideals. Denote $M = G/K$ and $o = eK \in M$. Under our assumption, the subgroup of elements of $G$ that act trivially on $M$ is precisely $Z(G)$. We also assume $Z(G)$ is finite, which is the same as to ask that $K$ is compact. In this case, $K$ is a maximal compact subgroup of $G$. If $\uppi_o \colon G \to M$ stands for the orbit map at $o$, then $d(\uppi_o)_e \colon \mk{g} \to T_oM$ restricts to an isomorphism $\mk{p} \isoto T_oM$, and this is an isomorphism between the adjoint representation of $K$ on $\mk{p}$ and the isotropy representation of $K$ on $T_oM$. Although any $K$-invariant inner product on $\mk{p}$ gives a $G$-invariant Riemannian metric on $M$ and turns $M$ into a symmetric space of noncompact type, we take it to be the restriction of the Killing form $B$ of $\mk{g}$ to $\mk{p}$ by default\footnote{This assumption is harmless if $M$ is irreducible (which will always be the case for us in all concrete examples) since any other $G$-invariant Riemannian metric on $M$ would then differ from ours by a constant conformal factor.}. We then have a morphism of Lie groups $G \twoheadrightarrow I^0(M)$, and this is a local isomorphism with kernel $Z(G)$. We endow $\mk{g}$ with the inner product $B_\uptheta(X,Y) = -B(X, \uptheta Y)$. This inner product coincides with $B$ on $\mk{p}$ and equals $-B$ on $\mk{k}$, and the Cartan decomposition is orthogonal with respect to both $B_\uptheta$ and $B$. Whenever we talk about lengths and angles in $\mk{g}$, it is with respect to $B_\uptheta$ unless otherwise specified. If we have subspaces $U \subset V \subseteq \mk{g}$, we let $V \ominus U$ stand for the orthogonal complement of $U$ in $V$. Furthermore, if we have a direct sum decomposition $V = U \oplus W$ such that $U$ and $W$ are orthogonal to each other, we sometimes stress it by writing $V = U \oplus^\perp W$. Observe that $\uptheta$ is orthogonal and self-adjoint with respect to both $B$ and $B_\uptheta$. The adjoint representation of $G$ is orthogonal with respect to $B$ but not $B_\uptheta$. The adjoint representation of $K$ on $\mk{g}$, however, is orthogonal with respect to $B_\uptheta$ because it preserve the Cartan decomposition and thus commutes with $\uptheta$.

Pick a maximal abelian subspace $\mk{a}$ of $\mk{p}$ and denote $r = \dim \mk{a} = \rk(M)$. Let $A \subseteq G$ be the connected abelian Lie subgroup with $\Lie(A) = \mk{a}$. The totally geodesic submanifold $A \ccdot o \simeq \mathbb{E}^r \subseteq M$ is a maximal flat. Having fixed $\mk{a}$, we have the corresponding root system $\Upsigma \subset \mk{a}^*$ and restricted root space decomposition $\mk{g} = \mk{g}_0 \oplus \bigoplus_{\upalpha \in \Upsigma} \mk{g}_\upalpha$. All summands in this decomposition are mutually orthogonal. We also have an orthogonal decomposition $\mk{g}_0 = \mk{k}_0 \oplus \mk{a}$, where $\mk{k}_0 = Z_\mk{k}(\mk{a}) = N_\mk{k}(\mk{a})$. The Dynkin diagram\footnote{We treat $\DD_M$ as a vertex-weighted graph where each vertex is assigned the multiplicity of the corresponding simple root. In particular, whenever we talk about automorphisms of $\DD_M$, we require them to preserve these multiplicities and denote the corresponding group by $\Aut^\mathrm{w}(\DD_M)$ ('w' for 'weighted').} of $\Upsigma$ will be denoted by $\mathrm{DD}_M$. Pick a subset of positive roots $\Upsigma^+ \subset \Upsigma$ with $\Uplambda = \set{\upalpha_1,\ldots,\upalpha_r} \subseteq \Upsigma^+$ the corresponding set of simple roots and write $\mk{n} = \bigoplus_{\upalpha \in \Upsigma^+} \mk{g}_\upalpha$. This is a nilpotent Lie subalgebra and it gives rise to a connected nilpotent Lie subgroup $N \subset G$. We have a vector space decomposition $\mk{g} = \mk{k} \oplus \mk{a} \oplus \mk{n}$ known as the Iwasawa decomposition and its global counterpart $G = KAN$ (which is just a diffeomorphism). In particular, both $A$ and $N$ are closed and simply connected and $AN$ is a solvable subgroup of $G$ acting simply transitively on $M$. Consequently, $M$ can be realized as a solvable subgroup of $G$ with a suitable left-invariant metric, although this of course depends on many choices made along the way. The orbit $N \ccdot o$ of $N$ is a properly embedded submanifold of $M$ called a horocycle.

For each $\upalpha \in \Upsigma$ we define 
\begin{align*}
    \mk{k}_\upalpha = \mk{k} \cap (\mk{g}_\upalpha \oplus \mk{g}_{-\upalpha}) = \set{X + \uptheta X \mid X \in \mk{g}_\upalpha}, \\
    \mk{p}_\upalpha = \mk{p} \cap (\mk{g}_\upalpha \oplus \mk{g}_{-\upalpha}) = \set{X - \uptheta X \mid X \in \mk{g}_\upalpha},
\end{align*}
so $\mk{k}_\upalpha \oplus \mk{p}_\upalpha = \mk{g}_\upalpha \oplus \mk{g}_{-\upalpha}$.

When restricted to $\mk{a}$, $B_\uptheta$ gives rise to an isomorphism $\mk{a} \isoto \mk{a}^*$. We carry $B_\uptheta$ along this isomorphism to get an inner product on $\mk{a}^*$. Given $\upalpha \in \Upsigma$, we write $H_\upalpha$ for the corresponding vector in $\mk{a}$, so $\cross{H_\upalpha}{H_\upbeta} = \bilin{H_\upalpha}{\upbeta} = \bilin{\upalpha}{H_\upbeta} = \cross{\upalpha}{\upbeta}$. We also let $H^1, \ldots, H^r$ stand for the basis of $\mk{a}$ dual to $\upalpha_1, \ldots, \upalpha_r$, so we have $\bilin{H_i}{\upalpha_j} = \cross{H^i}{H_{\upalpha_j}} = \updelta_{ij}$.

Conjugacy classes of parabolic subalgebras of $\mk{g}$ are parametrized by subsets of $\Uplambda$ modulo the action of $\Aut^\mathrm{w}(\mathrm{DD}_M)$ and thus allow description within the framework of the root space decomposition. We will confine ourselves to maximal proper parabolic subalgebras only, i.e. those corresponding to subsets of cardinality $r-1$. Take any $j \in \set{1,\ldots,r}$, and denote $\Upphi_j = \Uplambda \mysetminus \set{\upalpha_j}$. We write $\Upsigma_j$ for the root subsystem of $\Upsigma$ generated by $\Upphi_j$ and $\Upsigma_j^+ = \Upsigma_j \cap \Upsigma^+$. Define $\mk{l}_j = \mk{g}_0 \oplus \bigoplus_{\upalpha \in \Upsigma_j} \mk{g}_\upalpha$ and $\mk{n}_j = \bigoplus_{\upalpha \in \Upsigma^+ \mysetminus \Upsigma_j^+} \mk{g}_\upalpha$. These subalgebras of $\mk{g}$ are reductive and nilpotent, respectively. One can easily check that $[\mk{l}_j, \mk{n}_j] \subseteq \mk{n}_j$, so $\mk{q}_j = \mk{l}_j \oplus \mk{n}_j$ is a Lie subalgebra of $\mk{g}$. In fact, it is a maximal proper parabolic subalgebra of $\mk{g}$, and the semidirect sum decomposition $\mk{q}_j = \mk{l}_j \loplus \mk{n}_j$ is called the Chevalley decomposition of $\mk{q}_j$. Every maximal proper parabolic subalgebra of $\mk{g}$ is $\Inn(\mk{g})$-conjugate to $\mk{q}_j$ for a unique $j \in \set{1,\ldots,r}$, and two subalgebras $\mk{q}_j$ and $\mk{q}_{i}$ are $\Aut(\mk{g})$-conjugate if and only if there exists an automorphism of $\DD_M$ mapping $\upalpha_j$ to $\upalpha_i$. Next, write $\mk{a}_j = \bigcap_{\upalpha \in \Upphi_j} \ker \upalpha = \Rl H^j$ and $\mk{a}^j = \mk{a} \ominus \mk{a}_j = \bigoplus_{\upalpha \in \Upphi_j} \Rl H_\upalpha$. It is straightforward to check that $\mk{l}_j = Z_{\mk{g}}(\mk{a}_j) = N_{\mk{g}}(\mk{a}_j)$. Define $\mk{m}_j = \mk{l}_j \ominus \mk{a}_j$. This is also a reductive Lie subalgebra, and the Chevalley decomposition subdivides further into $\mk{q}_j = (\mk{m}_j \oplus \mk{a}_j) \loplus \mk{n}_j \defqe \mk{m}_j \oplus \mk{a}_j \loplus \mk{n}_j$, which is called the Langlands decomposition of $\mk{q}_j$. Let us write $\mk{m}_j = \mk{z}_j \oplus \mk{g}_j$ for the (unique) Levi decomposition of $\mk{m}_j$. Here $\mk{z}_j = \mk{z}(\mk{m}_j) \subseteq \mk{k}_0$, while $\mk{g}_j = [\mk{m}_j, \mk{m}_j] = [\mk{l}_j, \mk{l}_j]$ is semisimple. One readily sees that $\mk{g}_j = (\mk{k}_0 \ominus \mk{z}_j) \oplus \mk{a}^j \oplus \bigoplus_{\upalpha \in \Upsigma_j} \mk{g}_\upalpha$ and that it is a $\uptheta$-stable subalgebra of $\mk{g}$. What is more, $\uptheta$ restricts to a Cartan involution of $\mk{g}_j$. We denote $\mk{b}_j = \mk{g}_j \cap \mk{p}$. Note that $\mk{a}^j$ is a maximal abelian subspace of $\mk{b}_j$ and the corresponding restricted root system for $\mk{g}_j$ is $\Upsigma_j$. Finally, we define $\mk{k}_j = \mk{q}_j \cap \mk{k} = \mk{m}_j \cap \mk{k}$.

Now we look at what this plethora of subalgebras and decompositions produces when lifted to $G$. First things first, we have simply connected closed Lie subgroups $A_j$ and $N_j$ of $G$ corresponding to $\mk{a}_j$ and $\mk{n}_j$, respectively. By design, $A_j$ is abelian and $N_j$ is nilpotent. Define $L_j = Z_G(\mk{a}_j)$. This is a (possibly disconnected) closed reductive Lie subgroup of $G$ with Lie algebra $\mk{l}_j$. The product $Q_j = L_j N_j \subseteq G$ is a closed Lie subgroup with Lie algebra $\mk{q}_j$, and this is a maximal proper parabolic subgroup of $G$. It is not hard to see that $L_j$ normalizes $N_j$ and they intersect trivially, so we actually have a semidirect product $Q_j = L_j \ltimes N_j$, which is the global version of the Chevalley decomposition. Next, let $G_j$ be the connected Lie subgroup corresponding to $\mk{g}_j$, write $K_j = L_j \cap K$, and define $M_j = K_j G_j$. These are all closed subgroups, $K_j$ is compact, $M_j$ is reductive, $\Lie(K_j) = \mk{k}_j$, and $\Lie(M_j) = \mk{m}_j$. We have inclusions $M_j \subset L_j \subset Q_j$, and one can show that $K_j$ is a maximal compact subgroup in all three of these groups. Finally, let $Z_j$ stand for the center of $M_j$. This is a compact subgroup of $K_j$ with Lie algebra $\mk{z}_j$. The subgroups $M_j$ and $A_j$ commute and intersect trivially, hence we have a direct product decomposition $L_j = M_j \times A_j$, which, when plugged into the global Chevalley decomposition, induces the global version of the Langlands decomposition: $Q_j = M_j \times A_j \ltimes N_j$.

Eventually, we turn our attention to $M$. The parabolic subgroup $Q_j$ acts transitively on $M$ with isotropy subgroup $K_j$ at $o$. The subgroups $A_j$ and $N_j$ give us orbits $A_j \ccdot \hspace{0.3pt} o \simeq \mb{E}$, which is a geodesic lying in the maximal flat $A \hspace{0.3pt} \ccdot o$, and $N_j \ccdot o$, which a properly embedded submanifold of the horocycle $N \hspace{-0.5pt} \ccdot o$. We have a Lie triple system $\mk{b}_j \subseteq \mk{p}$, which corresponds to a totally geodesic properly embedded submanifold $B_j = G_j \ccdot o = M_j \ccdot o \subseteq M$ often called a boundary component of $M$ in the context of the maximal Satake compactification of $M$ (see, e.g., \cite{borel_ji}). The submanifold $B_j$ is itself a symmetric space of noncompact type. It has rank $r-1$ and can be represented, for example, by the Riemannian symmetric pair $(M_j^0,K_j^0)$. We also have a totally geodesic submanifold $F_j = L_j \ccdot o \simeq B_j \times (A_j \ccdot o)$ corresponding to the Lie triple system $\mk{b}_j \oplus \mk{a}_j \subseteq \mk{p}$ (note that $F_j$ is a symmetric space and $F_j \simeq B_j \times (A_j \ccdot o)$ is its decomposition into the Riemannian product of its noncompact and Euclidean parts). Finally, we can form a commutative diagram
$$
\xymatrix{
M_j \times A_j \times N_j \ar[r]^-{\sim} \ar@{->>}[d] & Q_j \ar@{->>}[d] \ar@{->>}[dr] \\
B_j \times A_j \times N_j \ar[r]^-{\sim} & Q_j/K_j \ar[r]^(0.57){\sim} & M .
}
$$
The diffeomorphism $B_j \times A_j \times N_j \simeq M$ is called a horospherical decomposition of $M$.

\subsection{Classes of cohomogeneity-one actions.}\label{classes_of_actions} Now we describe the different types of cohomogeneity-one actions that appear in the structure results of Berndt and Tamaru.

Let $\ell \subseteq \mk{a}$ be a one-dimensional subspace. Then the connected Lie subgroup $H_\ell$ of $G$ with Lie algebra $\mk{h}_\ell = (\mk{a} \ominus \ell) \oplus \mk{n}$ is closed and its action on $M$ has cohomogeneity one and no singular orbits (in other words, its orbits form a Riemannian foliation of codimension one). Moreover, all of its orbits are isometrically congruent to each other. If $T$ is an automorphism of the Dynkin diagram $\DD_M$ and $\widehat{T}$ is the orthogonal transformation of $\mk{a}$ induced by $T$, then the actions of $H_\ell$ and $H_{\widehat{T}(\ell)}$ are orbit equivalent. Conversely, if $\ell$ and $\ell'$ are one-dimensional subspaces of $\mk{a}$ such that the actions of $H_\ell$ and $H_{\ell'}$ are orbit equivalent, then these subspaces differ by $\widehat{T}$ for some $T \in \Aut^\mathrm{w}(\DD_M)$. See \cite{berndttamarufoliations} and \cite{solonenko2021homogeneous} for more details.

Let $\ell \subseteq \mk{g}_{\upalpha_i}$ be a one-dimensional subspace, $1 \leqslant i \leqslant r$. Then the connected Lie subgroup $H_i$ of $G$ with Lie algebra $\mk{h}_i = \mk{a} \oplus (\mk{n} \ominus \ell)$ is closed and its action on $M$ has cohomogeneity one, no singular orbits, and exactly one minimal orbit. Were we to choose another line $\ell'$ in $\mk{g}_{\upalpha_i}$, the resulting action would be orbit equivalent to that for $\ell$. If $T \in \Aut^\mathrm{w}(\DD_M)$ is some automorphism mapping $\upalpha_i$ to $\upalpha_j$, the actions of $H_i$ and $H_j$ are orbit equivalent. Conversely, if we have $1 \leqslant i,j \leqslant r$ such that the actions of $H_i$ and $H_j$ are orbit equivalent, then $\upalpha_i$ and $\upalpha_j$ differ by some automorphism of $\DD_M$. Again, see \cite{berndttamarufoliations} and \cite{solonenko2021homogeneous} for details.

Let $L$ be a maximal proper connected reductive Lie subgroup of $G$ (it is automatically closed). Let $H \subseteq L$ be a closed connected subgroup acting on $M$ with cohomogeneity one. Then $H$ and $L$ have the same orbits (and thus are trivially orbit equivalent), and their orbit through $o$ is totally geodesic. Moreover, unless $M$ is isometric to a real hyperbolic space, this orbit is singular. Such actions were explicitly classified in \cite{berndttamarusingtotgeod} for $M$ irreducible. In a nutshell, the classification goes as follows. A properly embedded connected submanifold $F$ of $M$ is called reflective if the geodesic reflection in $F$ is a well-defined isometry of $M$ (equivalently, if $F$ is a connected component of the fixed point set of an involutive isometry of $M$). A reflective submanifold $F$ is necessarily totally geodesic and any of its normal spaces is itself tangent to a totally geodesic submanifold, which we sloppily denote\footnote{It depends on the choice of a point in $F$ but only up to congruence in $M$.} by $F^\perp$. If $F$ is reflective and $F^\perp$ is of rank 1 (hence a hyperbolic space), then $F$ is a totally geodesic singular orbit of a cohomogeneity-one action. Conversely, a totally geodesic singular orbit $F$ of a cohomogeneity-one action is reflective with $F^\perp$ of rank one, apart from 5 non-reflective examples, all of which are mysteriously related to the exceptional Lie group $G_2$ (and none of which will appear in this paper so they are of no concern to us). What Berndt and Tamaru did in \cite{berndttamarusingtotgeod} is they classified reflective submanifolds $F$ with $\rk(F^\perp) = 1$ for each irreducible $M$ of noncompact type.

Now fix $j \in \set{1,\ldots,r}$. Let $H_j$ be a closed connected subgroup of $I(B_j)$ acting on the boundary component $B_j$ with cohomogeneity one. Write $\mk{i}(B_j)$ for the Lie algebra of $I(B_j)$ and $\mk{h}_j \subseteq \mk{i}(B_j)$ for that of $H_j$. Note that as a symmetric space, $B_j$ can be represented by a Riemannian symmetric pair $(G_j, G_j \cap K)$. In particular, we have a morphism $G_j \twoheadrightarrow I^0(B_j)$, which induces a Lie algebra homomorphism $\mk{g}_j \twoheadrightarrow \mk{i}(B_j)$. It may fail to be an isomorphism, for the action $G_j \curvearrowright B_j$ does not have to be effective. Nonetheless, $\mk{g}_j$ is semisimple, so there is a unique ideal in it complementary to the kernel of the map $\mk{g}_j \twoheadrightarrow \mk{i}(B_j)$, hence we can identify $\mk{i}(B_j)$ with this ideal (see Subsection \ref{error correction} for a more explicit description of this ideal and discussion of the non-effectiveness of $G_j \curvearrowright B_j$). But now we have $\mk{h}_j \subseteq \mk{i}(B_j) \hookrightarrow \mk{g}_j \subseteq \mk{m}_j$ as well as the Langlands decomposition $\mk{q}_j = \mk{m}_j \oplus \mk{a}_j \loplus \mk{n}_j$, so we can form a Lie subalgebra $\mk{h}_j^\Uplambda = \mk{h}_j \oplus \mk{a}_j \loplus \mk{n}_j \subseteq \mk{q}_j$. The corresponding connected Lie subgroup $H_j^\Uplambda \subseteq Q_j$ is closed and acts on $M$ with cohomogeneity one. This action is called the canonical extension of the action $H_j \curvearrowright B_j$, and the group $H_j^\Uplambda$ itself is called the canonical extension of $H_j$. The action of $H_j^\Uplambda$ on $M$ has a singular orbit if and only if the action of $H_j$ on $B_j$ has one, and the codimension of the former in $M$ equals the codimension of the latter in $B_j$. Note that if we are given two such subgroups $H_j, H'_j \subseteq I^0(B_j)$ whose actions on $B_j$ are orbit equivalent, the resulting actions of $H_j^\Uplambda$ and ${H'_j}^\Uplambda$ on $M$ may fail to be orbit equivalent (see the example after Proposition 4.2 in \cite{berndttamarucohomogeneityone}). However, if the actions of $H_j$ and $H'_j$ are orbit equivalent by some $f \in I^0(B_j)$ (and not just $I(B_j)$), the actions of $H_j^\Uplambda$ and ${H'_j}^\Uplambda$ will in fact be orbit equivalent. For this reason we introduce the following useful notion: two isometric actions on a Riemannian manifold $Q$ are called \textbf{strongly orbit equivalent} if their orbits can be identified by some isometry $f \in I^0(Q)$. We stress that $I^0(B_j)$ is not a subgroup of $G_j$ or $M_j$ but a quotient thereof, so it is necessary to embed $\mk{i}(B_j)$ into $\mk{g}_j$ for the above construction to work, as the sum $\mk{h}_j \oplus \mk{a}_j \loplus \mk{n}_j$ does not make sense a priori (we are pointing out this small issue because it was not taken into account in the description of the canonical extension method in \cite{berdntdominguez-vazquez} and \cite{berndttamarucohomogeneityone}). We refer to \cite{berndttamarucohomogeneityone} for more details.

Finally, we describe the nilpotent construction method. Once again, fix $j \in \set{1,\ldots,r}$. The nilpotent Lie algebra $\mk{n}_j$ is graded: 
$$
\mk{n}_j = \bigoplus_{\upnu = 1}^k \mk{n}_j^\upnu, \quad \mk{n}_j^\upnu = \bigoplus_{\upalpha(H^j) = \upnu} \mk{g}_\upalpha, \quad k = \updelta(H^j),
$$
where $\updelta \in \Upsigma^+$ is the highest root. The adjoint actions of $L_j$ on $\mk{n}_j$ respects this grading. Let $\mk{v} \subseteq \mk{n}_j^1$ be a linear subspace of dimension at least 2. One can easily see that $\mk{n}_{j,\mk{v}} = \mk{n}_j \ominus \mk{v}$ is a Lie subalgebra. Its corresponding connected Lie subgroup of $G$ is closed and will be denoted by $N_{j, \mk{v}}$. We also have $N_{\mk{l}_j}(\mk{n}_{j,\mk{v}}) = N_{\mk{m}_j}(\mk{n}_{j,\mk{v}}) \oplus \mk{a}_j$, which implies $N_{L_j}(\mk{n}_{j,\mk{v}}) = N_{M_j}(\mk{n}_{j,\mk{v}}) \times A_j$. Furthermore, $N_{\mk{l}_j}(\mk{n}_{j,\mk{v}}) = \uptheta N_{\mk{l}_j}(\mk{v})$ and thus $N^0_{L_j}(\mk{n}_{j,\mk{v}}) = \Uptheta N^0_{L_j}(\mk{v})$, and the same remains true with $\mk{l}_j$ and $L_j$ replaced with $\mk{m}_j$ and $M_j$, respectively. Define a closed connected Lie subgroup $H_{j, \mk{v}} = N^0_{L_j}(\mk{n}_{j,\mk{v}}) \ltimes N_{j,\mk{v}} \subset L_j \ltimes N_j = Q_j$. It was shown in \cite{berdntdominguez-vazquez} and \cite{berndttamarucohomogeneityone} that the following assumptions are equivalent:

\begin{enumerate}[(i)]
    \item $F_j \subseteq H_{j, \mk{v}} \ccdot o$.
    \item $N^0_{L_j}(\mk{n}_{j,\mk{v}})$ acts transitively on $F_j$.
    \item $N^0_{M_j}(\mk{n}_{j,\mk{v}})$ acts transitively on $B_j$.
    \item The image of the projection of $N_{\mk{l}_j}(\mk{v})$ to $\mk{p}$ along $\mk{k}$ equals $\mk{b}_j \oplus \mk{a}_j$.
    \item The image of the projection of $N_{\mk{m}_j}(\mk{v})$ to $\mk{p}$ along $\mk{k}$ equals $\mk{b}_j$.
\end{enumerate}

Now, suppose the following two assumptions, which we sometimes call 'assumptions (1) and (2) of the nilpotent construction', are satisfied:

\begin{enumerate}
    \item The image of the projection of $N_{\mk{m}_j}(\mk{v})$ to $\mk{p}$ along $\mk{k}$ equals $\mk{b}_j$, and
    \item $N_{K_j}(\mk{n}_{j,\mk{v}}) = N_{K_j}(\mk{v})$ acts transitively on the unit sphere in $\mk{v}$.
\end{enumerate}

In this case, $H_{j, \mk{v}}$ acts on $M$ with cohomogeneity one and its orbit though $o$ is singular of codimension equal to $\dim \mk{v}$. What is more, if $\mk{v}' \subseteq \mk{n}_j^1$ is another such subspace that differs from $\mk{v}$ by $\Ad(k)$ for some $k \in K_j$, then $\mk{v}$ satisfies the above assumptions if and only if $\mk{v}'$ does, and, if they do, the actions $H_{j, \mk{v}}$ and $H_{j, \mk{v}'}$ are orbit equivalent. See \cite{berdntdominguez-vazquez} and \cite{berndttamarucohomogeneityone} for details.

Now we formulate the main result from \cite{berndttamarucohomogeneityone}, which guarantees that all cohomogeneity-one actions on irreducible symmetric spaces of noncompact type can be obtained by one of the five methods described above.

\begin{theorem}\label{maintheorem}
Let $M = G/K$ be an irreducible Riemannian symmetric space of noncompact type and rank $r$, and let $H$ be a closed connected subgroup of $G$ acting on $M$ with cohomogeneity one. Then one of the following statements holds:

\begin{enumerate}[\normalfont (1)]
    \item The orbits of $H$ form a Riemannian foliation and exactly one of the following two cases holds:
    
    \begin{enumerate}[\normalfont (i)]
        \item The action of $H$ is orbit equivalent to the action of $H_\ell$ for some one-dimensional linear subspace $\ell \subseteq \mk{a}$, where $\ell$ is determined uniquely up to the action of $\Aut^\mathrm{w}(\DD_M)$ on $\mb{P}\mk{a}$.
        \item The action of $H$ is orbit equivalent to the action of $H_i$ for some $i \in \set{1,\ldots,r}$, where $i$ is determined uniquely up to the action of $\Aut^\mathrm{w}(\DD_M)$ on $\set{1,\ldots,r}$.
    \end{enumerate}
    
    \item There exists exactly one singular orbit and one of the following two cases holds:
    
    \begin{enumerate}[\normalfont (i)]
        \item $H$ is contained in a maximal proper connected reductive Lie subgroup $L$ of $G$, $H$ and $L$ have the same orbits, and the singular orbit is totally geodesic and isometrically congruent to one of the submanifolds described in Theorems {\normalfont 3.3} and {\normalfont 4.2} in {\normalfont \cite{berndttamarusingtotgeod}}.
        \item $H$ is contained in a maximal proper parabolic subgroup $Q_j$ of $G$ (for suitable choices of $o \in M, \hspace{1pt} \mk{a} \subseteq \mk{p},$ and $\Upsigma^+ \subset \Upsigma$) for some $j \in \set{1, \ldots, r}$ and one of the following two subcases holds:
        
        \begin{enumerate}[\normalfont (a)]
            \item The action of $H$ is orbit equivalent to the canonical extension of a cohomogeneity-one action with a singular orbit on the boundary component $B_j$ of $M$.
            \item The action of $H$ is orbit equivalent to the action of a group $H_{j, \mk{v}}$ obtained by nilpotent construction applied to some subspace $\mk{v} \subseteq \mk{n}_j^1$ with $\dim \mk{v} \geqslant 2$ satisfying both assumptions of the construction.
        \end{enumerate}
    \end{enumerate}
\end{enumerate}
\end{theorem}

\begin{remarks}
\begin{enumerate}[1.]
    \item Here is how one goes about explicitly classifying cohomogeneity-one actions on a given irreducible symmetric space of noncompact type. Actions without singular orbits and with a totally geodesic singular orbit correspond to (1) and (2)(i), respectively, and we already have their explicit classification, so we only need to classify actions with a non-totally-geodesic singular orbit. To this end, a priori, we need to find all actions arising through the canonical extension and nilpotent construction methods for each $j \in \set{1,\ldots,r}$. However, the problem is alleviated by the fact that if there exists an automorphism of $\DD_M$ mapping $\upalpha_i$ to $\upalpha_j$, then we need only deal with these two constructions for $i$. This follows from the fact that any $T \in \Aut^\mathrm{w}(\DD_M)$ gives an automorphism of $\Upsigma$ as a root system, and any automorphism of $\Upsigma$ is given by $(\restr{\Ad(k)}{\mk{a}})^*$, where $k \in N_{\widetilde{K}}(\mk{a})$ and $\widetilde{K}$ is the isotropy subgroup of $I(M)$ at $o$. If we start with $T \in \Aut^\mathrm{w}(\DD_M)$ mapping $\upalpha_i$ to $\upalpha_j$ and find $k \in N_{\widetilde{K}}(\mk{a})$ that gives the corresponding automorphism of $\Upsigma$, then $k$ maps $B_j$ onto $B_i$, $\mk{n}_j$ onto $\mk{n}_i$, etc., so it provides an orbit equivalence between actions canonically extended from $B_j$ and $B_i$, and the same for the nilpotent construction (this argument remains valid if $M$ is reducible provided that the Riemannian metric comes from the Killing form). Finding all actions arising by canonical extension from $B_j$ entails knowing the classification of cohomogeneity-one actions on $B_j$ up to strong orbit equivalence. Since $B_j$ has rank lower than $M$, an inductive algorithm is sometimes possible (the main problem here is that $B_j$ may happen to be reducible\footnote{Precisely when removal of $\upalpha_j$ makes $\DD_M$ disconnected.}, and the classification of cohomogeneity-one actions on reducible spaces has not been fully developed yet). Finding all actions arising from the nilpotent construction basically boils down to a problem in representation theory of reductive Lie groups. We will discuss this in a bit more detail in the next subsection. Having found all such actions for all $j$, one also needs to tell which of them are orbit equivalent or have a totally geodesic singular orbit.
    \item Cases (2)(i), (2)(ii)(a), and (2)(ii)(b) may overlap. For example, an action obtained by nilpotent construction for some $j$ may have a totally geodesic singular orbit or be orbit equivalent to some action canonically extended from some $B_i$. We will see examples of such behavior later in the paper.
    \item It is usually fairly straightforward to find what the Lie algebras $\mk{m}_j$ and $\mk{k}_j$ and thus the groups $M_j^0$ and $K_j^0$ look like. It may be a bit harder, however, to do so for the full groups $M_j$ and $K_j$. In practical terms, it means that in the nilpotent construction it may be easier to find all subspaces $\mk{v} \subseteq \mk{n}_j^1$ satisfying the two assumptions up to $K_j^0$-congruence but not $K_j$-congruence. To make things worse, even if two subspaces $\mk{v}, \mk{v}' \subseteq \mk{n}_j^1$ are not $K_j$-congruent, they may still, in theory, produce orbit equivalent actions, so we may end up with a larger list of actions and have to discard some of them. However, this scarcely poses an actual problem, for the reality is that whenever we see an action arising by nilpotent construction, chances are it either has a totally geodesic singular orbit or arises by canonical extension from some boundary component of $M$.
    \item Finally, note that if $M$ has rank 1 and is, therefore, isometric to a hyperbolic space over $\Rl, \Cx, \Hq,$ or $\mb{O}$, then it has $\Uplambda = \set{\upalpha_1}$ and the boundary component $B_1$ is a point. For this reason, the canonical extension method produces exactly one cohomogeneity-one action, which is nothing but the action of the isotropy subgroup $K$. Moreover, assumption (1) of the nilpotent construction is trivially satisfied for every $\mk{v} \subseteq \mk{n}_1^1$. Historically, the rank-1 case had mostly been handled before the advent of the canonical extension and nilpotent construction methods (see \cite{berndt_bruck} and \cite{berndttamarurankone}), although the case of $\Hq H^n$ stayed unresolved until recently (see \cite{protohomogeneous}).
\end{enumerate}
\end{remarks}

Theorem \ref{maintheorem} gives the classification of connected properly embedded homogeneous hypersurfaces in irreducible symmetric spaces of noncompact type up to isometric congruence, which justifies the title of the present paper. Indeed, given a connected complete Riemannian manifold $M$, connected properly embedded homogeneous hypersurfaces in $M$ are the same as nonsingular orbits of proper isometric cohomogeneity-one actions on $M$ by connected Lie groups. An orbit equivalence between two such actions provides an isometric congruence between their corresponding orbits by design. Conversely, let $H$ be a connected Lie group acting properly, isometrically, and with cohomogeneity one on $M$, and let $F \subseteq M$ be any orbit of $H$. Then all the other orbits are uniquely determined by $F$ as the connected components of the sets $F_r = \set{x \in M \mid \dist(x,F) = r}, \; r > 0$. (Informally, these are equidistant 'tubes' of various radii around $F$; $F_r$ is connected when $F$ is nonprincipal but may have two connected components otherwise.) Consequently, if we are given another such action $H' \curvearrowright M$ and an orbit $F' \subseteq M$ of $H'$, an isometry $f \in I(M)$ mapping $F$ onto $F'$ must map $F_r$ onto $F'_r$ for each $r > 0$, so $f$ is an orbit equivalence between the actions of $H$ and $H'$. If $F$ is a singular orbit and there are no other singular orbits (which is always the case when $M$ is a noncompact symmetric space, see \cite[Section 2]{berndt_bruck}), any orbit equivalence of the action of $H$ with itself preserves $F$ and thus each $F_r$, which means that no two distinct orbits of $H$ are congruent. Finally, considering the case of cohomogeneity-one actions without singular orbits, that is, cases (1)(i) and (1)(ii) in Theorem \ref{maintheorem}, it was shown in \cite{berndttamarufoliations} that all orbits of $H_\ell$ are congruent, whereas $H_i$ has a unique minimal orbit $F$, and $F_r$ consists of two orbits (for each $r>0$) that are congruent to each other and to no other orbits. Note that, given a cohomogeneity-one action $H \curvearrowright M$ and its principal orbit $F \subseteq M$, the singular points of the action are precisely the focal points of $F$.

Altogether, the classification of homogeneous hypersurfaces in an irreducible noncompact symmetric space $M$ goes like this. Every connected properly embedded homogeneous hypersurface in $M$ is isometrically congruent to a principal orbit of a cohomogeneity-one action on $M$ of one of the 5 types described in Theorem \ref{maintheorem}. Congruent hypersurfaces correspond to orbit equivalent actions. Each $H_\ell$ ($\ell \in \mb{P}\mk{a}/\Aut^\mathrm{w}(\DD_M)$) in (1)(i) gives precisely one congruence class of homogeneous hypersurfaces, each $H_i$ ($i \in \set{1, \ldots, r}\hspace{-1.5pt}/\Aut^\mathrm{w}(\DD_M)$) in (1)(ii) gives a one-parameter family of classes parametrized by $t \geqslant 0$, and each action in (2) gives a one-parameter family parametrized by $t>0$.

\subsection{Generalizing the nilpotent construction problem.}\label{nilpotent construction} Since the nilpotent construction is so involved and presents one of the main hindrances to the classification of cohomogeneity-one actions on irreducible symmetric spaces of noncompact type, it is worth looking at it from a slightly more general perspective and introducing some new terminology to talk about it more easily.

Let $G$ be a reductive Lie group with Lie algebra $\mk{g}$, a maximal compact subgroup $K \subseteq G$, and a Cartan decomposition $\mk{g} = \mk{k} \oplus \mk{p}$ (we are using the notation from \cite[Section VII.2]{knapp} in this subsection). Let $V$ be a finite-dimensional real representation of $G$. Fix some $K$-invariant inner product on $V$ and let $\mk{v} \subseteq V$ be a subspace with $\dim \mk{v} \geqslant 2$. Inspired by \cite{protohomogeneous}, we call $\mk{v}$ \textbf{protohomogeneous} if there exists a subgroup of $K$ preserving $\mk{v}$ and acting transitively on the unit sphere in $\mk{v}$. We call $\mk{v}$ \textbf{admissible} if the projection of $N_{\mk{g}}(\mk{v})$ to $\mk{p}$ along $\mk{k}$ is onto, i.e. its image equals the whole $\mk{p}$. The problem, which we call \textbf{the generalized nilpotent construction problem}, is to classify subspaces of $V$ that are both protohomogeneous and admissible up to the action of $K$. One readily sees how this generalizes the nilpotent construction problem by taking $G = M_j$ (or $L_j$), $K = K_j, \hspace{1.5pt} \mk{g} = \mk{m}_j = \mk{k}_j \oplus \mk{b}_j$ (or $\mk{l}_j = \mk{k}_j \oplus (\mk{b}_j \oplus \mk{a}_j)$), and $V = \mk{n}_j^1$. Admissibility and protohomogeneity become assumptions (1) and (2) of the nilpotent construction, respectively. Note that every subspace of $V$ is trivially admissible when $G$ is compact.

This problem is solved for a few representations. For instance, Berndt and Br\"{u}ck handled the cases $(G, V) = (\SO(n), \Rl^n), \hspace{1pt} (\U(n), \Cx^n), \hspace{1pt} (\Spin(7), \text{the spin representation} \; \Rl^8)$ in \cite{berndt_bruck}, while D\'{i}az-Ramos et al. dealt with $(G,V) = (\Sp(n)\Sp(1), \Hq^n)$ in \cite{protohomogeneous}. See also \cite{berdntdominguez-vazquez} and \cite{berndttamarucohomogeneityone} for some examples with $G$ noncompact. The latter of these two papers contains also a general result on the nilpotent construction problem that facilitates the search for admissible and protohomogeneous subspaces for certain noncompact symmetric spaces and choices of $j$ (see Proposition 5 there), and it may, in theory, be possible to extend this result to the generalized nilpotent construction problem. When dealing with the nilpotent construction on some rank-2 spaces later in the paper, we will solve this problem for a few more representations. We want to stress, however, that almost all attempts to solve the problem have so far been ad-hoc, and it does not seem feasible to proceed like this in the future. The more complex the representation of $G$ on $V$ is, the more difficult the problem becomes (for instance, it took the authors of \cite{protohomogeneous} an entire paper to solve the problem for $\Sp(n)\Sp(1) \curvearrowright \Hq^n$!). We believe that a more holistic and general approach to the problem is necessary if one wants to solve the nilpotent construction problem for all symmetric spaces of noncompact type.

\subsection{Error correction.}\label{error correction} Before proceeding to the classification of cohomogeneity-one actions on concrete symmetric spaces, we would like to point out a minor error in \cite[Paragraph 13.2]{submanifoldsholonomy}, \cite[Section 3]{hyperpolarfoliations}, and \cite[Section 2]{berndttamarucohomogeneityone} and correct it.

It is stated in the above works that $\mk{g}_j \cap \mk{k} = [\mk{b}_j, \mk{b}_j]$. It is not true in general. For example, if we take $\mk{g} = \mk{sl}(3, \mathbb{H})$ with $\uptheta(X) = X^* = -\overbar{X}$, $\mk{a} = \set{\text{real traceless diagonal matrices}}$, and $\upalpha_j$ any of the two simple roots of $\Upsigma$, then $\mk{g}_j \simeq \mk{sl}(2, \mathbb{H}) \hspace{0.5pt} \oplus \hspace{0.8pt} \mk{sp}(1)$, so any Cartan involution on it will be trivial on the compact summand $\mk{sp}(1)$. Consequently, we have $[\mk{b}_j, \mk{b}_j] \oplus \mk{b}_j = \mk{sl}(2, \mathbb{H})$ but $\mk{g}_j \cap \mk{k} = [\mk{b}_j, \mk{b}_j] \oplus \mk{sp}(1)$. The reason for this issue is that in general $\mk{g}_j$ may have nonzero compact ideals. In fact, we have the following

\begin{proposition}\label{errorcorrectionproposition}
Let $\mk{g}$ be a semisimple Lie algebra with a Cartan involution $\uptheta$, and let $\mk{g} = \bigoplus_{\upmu = 1}^k \mk{g}_\upmu \oplus \bigoplus_{\upnu = k+1}^n \mk{g}_\upnu$ be its decomposition into simple ideals, where the first $k$ ideals are noncompact and the rest are compact. Then $\uptheta$ respects this decomposition and is the identity on the compact ideals. Moreover, if we write $\mk{g}_\upmu = \mk{k}_\upmu \oplus \mk{p}_\upmu$ for the induced Cartan decomposition of each factor, then \vspace{-0.2em}
$$
[\mk{p}, \mk{p}] = \bigoplus_{\upmu=1}^k [\mk{p}_\upmu, \mk{p}_\upmu] = \bigoplus_{\upmu=1}^k \mk{k}_\upmu.
$$
\end{proposition}
\vspace{-1.5em}
\begin{proof}
The only thing which is not straightforward here is that $[\mk{p}_\upmu, \mk{p}_\upmu] = \mk{k}_\upmu$ for $1 \leqslant \upmu \leqslant k$. One possible way to show this is explained in \cite[Section VI, problems 22-24]{knapp}.
\end{proof}

Here and in the future we refer to the sum of all (non)compact ideals of a semisimple Lie algebra as its \textbf{(non)compact part}.

\begin{corollary}\label{compactpart}
In proposition {\normalfont \ref{errorcorrectionproposition}}, the compact part $\bigoplus_{\upnu = k+1}^n \mk{g}_\upnu$ of $\mk{g}$ can be described as $Z_\mk{k}(\mk{p}) = Z_\mk{g}(\mk{p})$.
\end{corollary}

\begin{proof}
For each $\upmu \in \set{1, \ldots, k}$, if $X$ lies in $Z_{\mk{g}_\upmu}(\mk{p}_\upmu)$, then it also commutes with $[\mk{p}_\upmu, \mk{p}_\upmu] = \mk{k}_\upmu$ and hence lies in the center of $\mk{g}_\upmu$, which is zero. The rest follows.
\end{proof}

This issue with $\mk{g}_j$ also means that the action of $G_j$ on the boundary component $B_j$ may fail badly to be effective: if there is a nonzero compact ideal in $\mk{g}_j$, its corresponding Lie subgroup will act trivially on $B_j$. So what we need is to exclude the compact ideals of $\mk{g}_j$ from consideration. We will give an explicit description of the noncompact part of $\mk{g}_j$ in terms of the restricted root space decomposition of $\mk{g}$.

Define $\boldsymbol{\widetilde{\mk{g}}_j}$ to be Lie subalgebra of $\mk{g}$ generated by $\mk{g}_\upalpha$'s with $\upalpha \in \Upsigma_j$. Clearly, $\widetilde{\mk{g}}_j \subseteq \mk{g}_j$. Since any $X \in \mk{g}_\upalpha$ satisfies $[X, \uptheta X] = -||X||^2 H_\upalpha$, we have $\mk{a}^j \subseteq \widetilde{\mk{g}}_j$, so we can write $\widetilde{\mk{g}}_j = (\widetilde{\mk{g}}_j \cap \mk{k}_0) \oplus \mk{a}^j \oplus \bigoplus_{\upalpha \in \Upsigma_j} \mk{g}_\upalpha$. Observe also that $(\widetilde{\mk{g}}_j \cap \mk{k}_0) \oplus \mk{a}^j = \sum_{\upalpha \in \Upsigma_j^+} [\mk{g}_\upalpha, \mk{g}_{-\upalpha}]$.

\begin{proposition}
The noncompact part of $\mk{g}_j$ is precisely $\widetilde{\mk{g}}_j$, while the compact part is $Z_{\mk{k}_0}(\mk{b}_j) \ominus \mk{z}_j$. The compact and noncompact parts of $\mk{g}_j$ are mutually orthogonal (with respect to $B_\uptheta$).
\end{proposition}

\begin{proof}
We start by looking at $Z_{\mk{k}_0}(\mk{b}_j)$ and proving a number of alternative expressions for it, namely $Z_{\mk{k}_0}(\mk{b}_j) = Z_{\mk{k}_0}(\bigoplus_{\upalpha \in \Upsigma_j} \mk{g}_\upalpha) = Z_{\mk{k}_0}(\widetilde{\mk{g}}_j) = Z_{\mk{m}_j}(\mk{b}_j) = Z_{\mk{m}_j}(\bigoplus_{\upalpha \in \Upsigma_j} \mk{g}_\upalpha) = Z_{\mk{m}_j}(\widetilde{\mk{g}}_j)$. It suffices to show that the smallest of these six centralizers, $Z_{\mk{k}_0}(\widetilde{\mk{g}}_j)$, coincides with the two largest ones, $Z_{\mk{m}_j}(\mk{b}_j)$ and $Z_{\mk{m}_j}(\bigoplus_{\upalpha \in \Upsigma_j} \mk{g}_\upalpha)$.

Take $X \in Z_{\mk{m}_j}(\mk{b}_j)$. Since $\mk{a}^j \subseteq \mk{b}_j$, $X$ cannot have nonzero components in any of $\mk{g}_\upalpha$, $\upalpha \in \Upsigma_j$, so $X \in \mk{k}_0 \oplus \mk{a}^j = \mk{g}_0 \cap \mk{m}_j$. Let $Y \in \mk{g}_\upalpha$, $\upalpha \in \Upsigma_j$. Then $Y - \uptheta Y \in \mk{b}_j$, so $[X, Y - \uptheta Y] = 0$, which means that $[X, Y] = [X, \uptheta Y]$. But $[X, Y] \in \mk{g}_\upalpha$, whereas $[X, \uptheta Y] \in \mk{g}_{-\upalpha}$, so both are zero, which implies that $X$ commutes with $\bigoplus_{\upalpha \in \Upsigma_j} \mk{g}_\upalpha$ and hence with the subalgebra generated by it, which is $\widetilde{\mk{g}}_j$. Finally, if $X$ had a nonzero component in $\mk{a}^j$, it would not commute with some $\mk{g}_\upalpha$, $\upalpha \in \Upsigma_j$. We conclude that $X \in Z_{\mk{k}_0}(\widetilde{\mk{g}}_j)$, so $Z_{\mk{m}_j}(\mk{b}_j) = Z_{\mk{k}_0}(\widetilde{\mk{g}}_j)$. Similar arguments show that $Z_{\mk{m}_j}(\bigoplus_{\upalpha \in \Upsigma_j} \mk{g}_\upalpha) = Z_{\mk{k}_0}(\widetilde{\mk{g}}_j)$.

Now we deal with the compact part of $\mk{g}_j$. According to Corollary \ref{compactpart}, it coincides with $Z_{\mk{g}_j}(\mk{b}_j)$. We claim that $Z_{\mk{k}_0}(\mk{b}_j)$ decomposes as an orthogonal direct sum $\mk{z}_j \oplus Z_{\mk{g}_j}(\mk{b}_j)$. Indeed, it follows immediately from the fact that $Z_{\mk{k}_0}(\mk{b}_j) = Z_{\mk{m}_j}(\mk{b}_j)$ and that $\mk{m}_j = \mk{z}_j \oplus^\perp \mk{g}_j$. The assertion in the proposition about the compact part follows.

Now we have to show that the noncompact part of $\mk{g}_j$ is $\widetilde{\mk{g}}_j$ and that it is orthogonal to the compact part. First, we show that $Z_{\mk{k}_0}(\mk{b}_j) \perp \widetilde{\mk{g}}_j$. Indeed, any $X$ from the centralizer is already orthogonal to $\bigoplus_{\upalpha \in \Upsigma_j} \mk{g}_\upalpha$, and if $Y, Z$ lie in the latter, we have $\cross{X}{[Y,Z]} = -B(X, [Y,Z]) = -B([X,Y],Z) = 0$. Since $\widetilde{\mk{g}}_j$ is spanned as a vector space by $\bigoplus_{\upalpha \in \Upsigma_j} \mk{g}_\upalpha$ and $[\bigoplus_{\upalpha \in \Upsigma_j} \mk{g}_\upalpha, \bigoplus_{\upalpha \in \Upsigma_j} \mk{g}_\upalpha]$, $X$ is orthogonal to $\widetilde{\mk{g}}_j$. Next, assume $X \in \mk{k}_0$ is orthogonal to $\widetilde{\mk{g}}_j$. We want to show that it lies in $Z_{\mk{k}_0}(\mk{b}_j)$, that is, commutes with all $\mk{g}_\upalpha$'s, $\upalpha \in \Upsigma_j$. Given $Y \in \mk{g}_\upalpha$, $Z = [X,Y] \in \mk{g}_\upalpha$, and $||Z||^2 = -B([X,Y], \uptheta Z) = -B(X, [Y, \uptheta Z]) = \cross{X}{[Y, \uptheta Z]} = 0$, since $[Y, \uptheta Z]$ lies in $\widetilde{\mk{g}}_j$. We deduce that $\mk{g}_j = (Z_{\mk{k}_0}(\mk{b}_j) \ominus \mk{z}_j) \oplus^\perp \widetilde{\mk{g}}_j$.

The fact that $\widetilde{\mk{g}}_j$ is actually an ideal of $\mk{g}_j$ is obvious because it has a complementary ideal $Z_{\mk{k}_0}(\mk{b}_j) \ominus \mk{z}_j$ with which it commutes.
\end{proof}

\begin{remark}
The same is true when we work with arbitrary parabolic subalgebras, not necessarily maximal ones. Instead of taking $\Upphi_j = \Uplambda \mysetminus \set{\upalpha_j}$, one takes arbitrary subsets $\Upphi \subseteq \Uplambda$ and defines $\mk{q}_\Upphi$ and other related subalgebras in a similar fashion. The semisimple subalgebra $\mk{g}_\Upphi$ may contain compact ideals, so the equality $[\mk{b}_\Upphi, \mk{b}_\Upphi] = \mk{g}_\Upphi \cap \mk{k}$ does not hold in general. But the above proof applies almost verbatim in this more general situation and yields a handy description of the compact and noncompact parts of $\mk{g}_\Upphi$.
\end{remark}

Define $\boldsymbol{\widetilde{G}_j}$ to be the connected Lie subgroup of $G$ corresponding to $\widetilde{\mk{g}}_j$. Since $\widetilde{G}_j$ is semisimple and $G$ has finite center, it follows from \cite[Section 6]{mostowsubgroups} that $\widetilde{G}_j$ is closed. The intersection $\widetilde{G}_j \cap K$ is then compact, so $(\widetilde{G}_j, \widetilde{G}_j \cap K)$ is a Riemannian symmetric pair of noncompact type whose corresponding symmetric space is $B_j$. It follows that $\widetilde{G}_j \cap K$ is connected and the center of $\widetilde{G}_j$ is finite and contained in $\widetilde{G}_j \cap K$. Therefore, we have a finite covering $\widetilde{G}_j \twoheadrightarrow \widetilde{G}_j/Z(\widetilde{G}_j) \cong I^0(B_j)$, so the action $G_j \curvearrowright B_j$ is 'almost' effective and $\widetilde{\mk{g}}_j \cong \mk{i}(B_j)$. This identification of $\mk{i}(B_j)$ with an ideal in $\mk{g}_j$ is precisely what we did in the canonical extension method. Note also that since we have $\mk{m}_j = Z_{\mk{k}_0}(\mk{b}_j) \oplus \widetilde{\mk{g}}_j$ and the first summand is precisely the kernel of the representation of $\mk{m}_j$ on $\mk{b}_j$, we can take the Lie algebra $\mk{h}_j^\Uplambda$ in the canonical extension to be not $\mk{h}_j \oplus \mk{a}_j \loplus \mk{n}_j$ but $\widehat{\mk{h}}_j \oplus \mk{a}_j \loplus \mk{n}_j$, where $\widehat{\mk{h}}_j \subseteq \mk{m}_j$ is any Lie subalgebra the image of whose projection to $\widetilde{\mk{g}}_j$ along $Z_{\mk{k}_0}(\mk{b}_j)$ is $\mk{h}_j$ and whose corresponding connected Lie subgroup of $M_j$ is closed (e.g., we can take $\widehat{\mk{h}}_j = Z_{\mk{k}_0}(\mk{b}_j) \oplus \mk{h}_j$). The resulting action on $M$ would still have the same orbits. This observation will let us describe canonically extended actions more neatly.

\subsection{Rank-one case.}\label{rank one case} Here we prove a small technical result pertaining to cohomogeneity-one actions on rank-1 noncompact symmetric spaces, i.e. on hyperbolic spaces over finite-dimensional real normed division algebras. It will allow us to describe actions on rank-2 spaces canonically extended from their boundary components in a relatively nice and uniform way. Note that if $M$ is a noncompact symmetric space of rank $1$, we have $\Uplambda = \set{\upalpha_1}$ and $\mk{m}_1 = \mk{k}_1 = \mk{k}_0$, so $M_1$ is compact and every subspace of $\mk{n}_1^1$ is automatically admissible.

\begin{proposition}\label{rank one proposition}
Let $M$ be a symmetric space of noncompact type and rank $1$. 

\begin{enumerate}[\normalfont (1)]
    \item If $\mk{v} \subseteq \mk{n}_1^1 = \mk{g}_{\upalpha_1}$ is a protohomogeneous subspace, then $\mk{h}_\mk{v} = N_\mk{k}(\mk{a} \oplus \mk{n}_{1,\mk{v}}) \oplus \mk{a} \oplus \mk{n}_{1,\mk{v}}$ is a Lie subalgebra of $\mk{g}$ whose corresponding connected Lie subgroup is closed, acts on $M$ with cohomogeneity one, and in fact has the same orbits as $H_{1, \mk{v}}$.
    \item Let $V \subseteq \mk{p}$ be a Lie triple system corresponding to a totally geodesic orbit $S$ of some cohomogeneity-one action on $M$, assume $\mk{a} \subseteq V$, and let $\mk{h}$ stand for the preimage of $V$ under the isomorphism $\mk{a} \oplus \mk{n} \isoto \mk{p}$. Then $N_\mk{k}(\mk{h}) \oplus \mk{h}$ is a Lie subalgebra of $\mk{g}$ whose corresponding connected Lie subgroup is closed, acts on $M$ with cohomogeneity one, and has $S$ as its orbit.
\end{enumerate}
\end{proposition}

\begin{proof}
For the first part, note that $\mk{h}_{1, \mk{v}} = N_{\mk{k}_0}(\mk{n}_{1,\mk{v}}) \oplus \mk{a} \oplus \mk{n}_{1,\mk{v}}$, so it is contained in $\mk{h}_\mk{v}$ (which is clearly a Lie subalgebra) and has the same projection in $\mk{p}$ as $\mk{h}_\mk{v}$. Since we already know that $H_{1, \mk{v}}$ acts with cohomogeneity one, it follows that its orbits coincide with those of the Lie subgroup corresponding to $\mk{h}_\mk{v}$. The latter is easily seen to be $N^0_K(A N_{1,\mk{v}}) A N_{1,\mk{v}}$, which is closed by the global Iwasawa decomposition.

For the second part, our assumption implies $\mk{a} \subseteq \mk{h}$, so we can write $\mk{h} = \mk{a} \oplus \mk{h}_n$, where $\mk{h}_n = \mk{h} \cap \mk{n}$. Assume for a moment that we know that $\mk{h}$ is a Lie subalgebra of $\mk{a} \oplus \mk{n}$ and write $H$ for its connected Lie subgroup. Clearly, $\mk{h}_n$ is then also a Lie subalgebra. Its corresponding Lie subgroup $H_n$ of $N$ is closed because the exponential map of $N$ is a diffeomorphism. We see that $H = AH_n$ and hence it is also a closed subgroup. Under our assumption, $N_\mk{k}(\mk{h}) \oplus \mk{h}$ is trivially a Lie subalgebra of $\mk{g}$ and its Lie subgroup is $N_K^0(H)H$, which is closed. Finally, $N_\mk{k}(\mk{h}) = N_\mk{k}(V)$, and the fact that $S$ is an orbit of some cohomogeneity-one action implies that the slice representation of $N^0_K(H)H$ at $o$ is of cohomogeneity one, hence so is the action of $N^0_K(H)H$ on $M$. 

It follows that we only need to prove that $\mk{h}$ is a Lie subalgebra of $\mk{a} \oplus \mk{n}$ and its corresponding Lie subgroup $H$ has $S$ as an orbit (a priori, we only know that its orbit through $o$ is of the same dimension and touches $S$ at $o$ tangentially), which we are going to do in a rather ad-hoc fashion. Consider first the special case when our totally geodesic orbit is 
\begin{equation}\label{totgeodrankone}
S' = \begin{cases}
\Rl H^k \; (0 \leqslant k \leqslant n-1), \quad &\text{if} \; M = \Rl H^n, \\
\Cx H^k \; (0 \leqslant k \leqslant n-1) \; \text{or} \; \Rl H^n, \quad &\text{if} \; M = \Cx H^n, \\
\Hq H^k \; (0 \leqslant k \leqslant n-1) \; \text{or} \; \Cx H^n, \quad &\text{if} \; M = \Hq H^n, \\
\Oo H^1 \; \text{or} \; \Hq H^2, \quad &\text{if} \; M = \Oo H^2,
\end{cases}
\end{equation}
where all these totally geodesic submanifolds are embedded into the corresponding hyperbolic space in a standard way. We are going to write $\mk{h}'$ and $\mk{h}'_n$ instead of $\mk{h}$ and $\mk{h}_n$ here to avoid notational confusion later on. Note that $S'$ is an orbit of the Lie subgroup
\begin{equation*}
G' = \begin{cases}
\SO^0(k,1) \; (0 \leqslant k \leqslant n-1), \quad &\text{if} \; M = \Rl H^n, \\
\SU(k,1) \; (0 \leqslant k \leqslant n-1) \; \text{or} \; \SO^0(n,1), \quad &\text{if} \; M = \Cx H^n, \\
\Sp(k,1) \; (0 \leqslant k \leqslant n-1) \; \text{or} \; \SU(n,1), \quad &\text{if} \; M = \Hq H^n, \\
\SO^0(8,1) \; \text{or} \; \Sp(2,1), \quad &\text{if} \; M = \Oo H^2,
\end{cases}
\end{equation*}
also embedded into the corresponding $G$ in a standard way. But $G'$ is $\Uptheta$-stable and $V$ coincides with $\mk{p}' = \mk{g}' \cap \mk{p}$, where $\mk{g}' = \Lie(G')$. Moreover, $\mk{a} \subseteq \mk{p}'$ is a maximal abelian subspace, while $\Upsigma' \cap \Upsigma^+$ is a choice of positive restricted roots for $\mk{g}'$. It all implies that $\mk{h}' = \mk{a} \oplus \mk{h}'_n$ coincides with $\mk{a} \oplus \mk{n}'$, where $\mk{n}'$ is the sum of positive root subspaces of $\mk{g}'$. But the latter sum is a Lie subalgebra of $\mk{a} \oplus \mk{n}$. What is more, since $S'$ is an orbit of $G'$, it is also an orbit of $H' = AH'_n = AN' \subseteq G'$, so $S'$ chosen in such a special way satisfies the assertions of (2).

Totally geodesic submanifolds in (\ref{totgeodrankone}) exhaust the list of totally geodesic orbits of cohomoge\-neity-one actions on hyperbolic spaces up to isometric congruence (see \cite[Theorem 1]{berndt_bruck}). In other words, given an arbitrary $S$ as in part (2) of the proposition, it is congruent to some $S'$ as in (\ref{totgeodrankone}), for which the assertions of (2) are satisfied. Now, $S'$ is a symmetric space of rank 1, so $I^0(S')$ acts homogeneously and isotropically on $S'$. Moreover, every element of $I^0(S')$ can be extended to an isometry of $M$. This is because $S'$ is a semisimple and totally geodesic complete submanifold of $M$ (see the argument in \cite[Proposition 3.2]{berndttamarusingtotgeod} or \cite[Theorem V.4.1(i)]{helgason}). It follows that we can take our congruence $k \in I(M)$ between $S$ and $S'$ to preserve $o$ and $\mk{a}$ and even fix $\mk{a}$ pointwise, i.e. to lie in $Z_{\widetilde{K}}(\mk{a})$, where $\widetilde{K}$ is the isotropy subgroup of $I(M)$ at $o$. But such $k$ preserves both $\mk{a} \oplus \mk{n}$ and $\mk{p}$ and is a Lie algebra automorphism of the former. Since it maps $V$ onto $V'$, it must map $\mk{h}$ onto $\mk{h}'$, so $\mk{h}$ is a Lie subalgebra. Moreover, $S'$ is an orbit of $H'$, which means that $S$ must be an orbit of $H$, which completes the proof.
\end{proof}

\section{Classification of cohomogeneity-one actions on \texorpdfstring{$\mathrm{SL}(3, \mathbb{H})/\mathrm{Sp}(3)$}{this space}}\label{first space}

In this section we classify, up to orbit equivalence, cohomogeneity-one actions on the noncompact symmetric space $\mathrm{SL}(3, \mathbb{H})/\mathrm{Sp}(3)$.

The symmetric space $M = \mathrm{SL}(3, \mathbb{H})/\mathrm{Sp}(3)$ is irreducible of rank 2 and dimension 14, and its root system is $A_2$. It is the quaternionic analog of the $A_2$-type spaces $\mathrm{SL}(3, \mathbb{R})/\SO(3)$ and $\mathrm{SL}(3, \Cx)/\SU(3)$, the explicit classification of cohomogeneity-one actions on which has already been obtained in \cite{berndttamarucohomogeneityone} and \cite{berdntdominguez-vazquez}, respectively. Let $\Uplambda = \set{\upalpha_1, \upalpha_2}$. Then $\Upsigma^+ = \set{\upalpha_1, \upalpha_2, \upalpha_1 + \upalpha_2}$ and all the roots have multiplicity 4. The Lie subalgebra $\mk{k}_0$ is isomorphic to $\mk{sp}(1)^{\oplus 3}$.

\begin{theorem}\label{SL(3,H)/Sp(3)}
Let $H \subset G = \SL(3, \mathbb{H})$ be a closed connected subgroup acting on $M$ with cohomogeneity one. Then its action is orbit equivalent to exactly one of the following:

\begin{enumerate}[\normalfont (1)]
    \item The action of the connected Lie subgroup $H_\ell$ of $G$ with Lie algebra
    $$
    \mk{h}_\ell = (\mk{a} \ominus \ell) \oplus \mk{n},
    $$
    where $\ell$ is a one-dimensional linear subspace of $\mk{a}$. Its orbits are all isometrically congruent to each other and form a Riemannian foliation on $M$.
    \item The action of the connected Lie subgroup $H_1$ of $G$ with Lie algebra
    $$
    \mk{h}_1 = \mk{a} \oplus (\mk{n} \ominus \ell_1),
    $$
    where $\ell_1$ is any one-dimensional linear subspace of $\mk{g}_{\upalpha_1}$. Its orbits form a Riemannian foliation on $M$ and there is exactly one minimal orbit.
    \item The action of the subgroup $L_1 = L_1^0$ of $G$. It has a 6-dimensional totally geodesic singular orbit $F_1 \simeq \mathbb{R}H^5 \times \mathbb{E}$.
    \item The action of the subgroup $\SL(3,\Cx)$ of $G$ embedded in a standard way. It has an 8-dimensional totally geodesic singular orbit isometric to $\mathrm{SL}(3, \Cx)/\SU(3)$.
    \item The action of the connected Lie subgroup $H^\Uplambda_{1,k}, k \in \set{0,1,2,3}$, of $G$ with Lie algebra
    $$
    \mk{h}_{1,k}^\Uplambda = N_{\mk{k}_1}(\mk{w}) \oplus \mk{w} \oplus \mk{a}_1 \oplus \mk{n}_1,
    $$
    where $\mk{w}$ is a $k$-dimensional subspace of $\mk{a}^1 \oplus \mk{g}_{\upalpha_2}$ containing $\mk{a}^1$ (unless $k = 0$). Here $N_{\mk{k}_1}(\mk{w}) \simeq \mk{sp}(1) \oplus \mk{so}(5-k) \oplus \mk{so}(k)$, where the first summand is $Z_{\mk{k}_0}(\mk{b}_1)$ and the rest is the normalizer of $\mk{w}$ in the Lie algebra $\widetilde{\mk{g}}_1 \cap \mk{k} \simeq \mk{so}(5)$ of the isotropy subgroup of the isometry group of the boundary component $B_1 \simeq \Rl H^5$. This action has a minimal\hspace{1.4pt}\footnote{Singular orbits of cohomogeneity-one actions are always minimal, see \cite[Proposition 3]{berndt_bruck}.} singular orbit of codimension $5-k$ and can be obtained by canonical extension of the cohomogeneity-one action on $B_1$ with $\Rl H^k$ as a totally geodesic singular orbit.
\end{enumerate}
\end{theorem}

\begin{proof}
We consider different cases of Theorem \ref{maintheorem}. If the orbits of $H$ form a foliation, we get the actions in (1) and (2). The actions in (3) and (4) are the only ones that have a totally geodesic singular orbit according to \cite{berndttamarusingtotgeod}. Here we need to check that the slice representations of $L_1$ and $\SL(3, \Cx)$ are transitive on spheres to make sure that these two groups do indeed act on $M$ with cohomogeneity one. It is not hard to show that $K_1 \simeq \Sp(2) \times \Sp(1)$ (in particular, $L_1$ is connected) and that the slice representation of $L_1$ at $o$ is equivalent to the standard representation of $\Sp(2) \times \Sp(1)$ on $\Hq^2$, which is transitive on the unit sphere. Modulo $\mb{Z}/2\mb{Z}$, this is the isotropy representation of the quaternionic hyperbolic plane, which should not be surprising because $F_1^\perp \simeq \Hq H^2$. What for $\SL(3, \Cx)$, its stabilizer at $o$ is $\SU(3)$ and its slice representation at $o$ can be easily computed to be of cohomogeneity one. In fact, it is equivalent to the tautological representation $\Cx^3$ of $\SU(3)$, which is just evidence of the fact that $(\SL(3,\Cx) \ccdot o)^\perp \simeq \Cx H^3$.

Now we determine the actions induced by canonical extension. Since the root system of $M$ is $A_2$, the boundary components $B_1$ and $B_2$ are isometrically congruent (see \cite[Proposition 28]{solonenko2022automorphisms}), so it suffices to consider only actions arising from $B_1$. The boundary component is isometric to $\SL(2, \mathbb{H})/\mathrm{Sp}(2) \simeq \SO^0(5,1)/\SO(5) \cong \Rl H^5$. There are, up to strong orbit equivalence, exactly 4 cohomogeneity-one actions with a singular orbit on $\Rl H^5$. According to Proposition \ref{rank one proposition} and the observation at the end of Subsection \ref{error correction}, the actions in (5) are the canonical extensions of these 4 actions on $B_1$. None of them are congruent to the actions in (3) and (4) because their singular orbits have different dimensions.

We are left to investigate actions arising from the nilpotent construction. Again, we need only consider one of the indices $\set{1,2}$, but this time we choose $j = 2$ (this choice is strategical as will become evident later on). We have $\mk{n}_2 = \mk{n}_2^1 = \mk{g}_{\upalpha_2} \oplus \mk{g}_{\upalpha_1 + \upalpha_2} \simeq \mathbb{H}^2$ and $K_2 \simeq \mathrm{Sp}(2) \times \mathrm{Sp}(1)$. The representation of $K_2$ on $\mk{n}_2$ is equivalent to the standard representation of $\Sp(2) \times \Sp(1)$ on $\mathbb{H}^2$:
$$
(A, \hspace{0.7pt} q) \ccdot \begin{bmatrix} x \\ y \end{bmatrix} = A \begin{bmatrix} xq^{-1} \\ yq^{-1} \end{bmatrix},
$$
where $A$ is a quaternionic-unitary $2 \times 2$ matrix.

We first filter out subspaces of $\mk{n}_2$ that are not protohomogeneous. This problem was recently solved in greater generality for the standard action of $\Sp(n)\Sp(1)$ on $\mb{H}^n$ by D\'{i}az-Ramos, Dom\'{i}nguez-V\'{a}zquez, and Rodr\'{i}guez-V\'{a}zquez in \cite{protohomogeneous}. The authors explicitly classified protohomogeneous subspaces of $\mb{H}^n$ up to the action of $\Sp(n)\Sp(1)$ in terms of their quaternionic K\"{a}hler angle -- the quaternionic analog of the K\"{a}hler angle of a real subspace of a Hilbert space first introduced by Berndt and Br\"{u}ck in \cite{berndt_bruck}. We recall its definition via the following lemma, which was essentially proven in \cite{berndt_bruck}:

\begin{lemma}
Let $\mk{v} \subseteq \mb{H}^n$ be a real subspace, and let $v \in \mk{v}$ be a nonzero vector. There exists a canonical basis $(J_1, J_2, J_3)$ of the quaternionic structure $\mathcal{J}$ on $\mb{H}^n$ and a uniquely defined triple $(\upvarphi_1, \upvarphi_2, \upvarphi_3) \in [0, \frac{\uppi}{2}]^3$ such that:

\begin{enumerate}[\normalfont (i)]
    \item $\upvarphi_i$ is the K\"{a}hler angle of $v$ with respect to $J_i$ (that is, the angle between $J_i v$ and $\mk{v}$) for each $i \in \set{1,2,3}$.
    \item $\cross{\pr_\mk{v} \circ J_i (v)}{\pr_\mk{v} \circ J_j (v)} = 0$ for each $i \ne j$, where $\pr_\mk{v}$ is the orthogonal projector onto $\mk{v}$ in $\mb{H}^n$.
    \item $\upvarphi_1 \leqslant \upvarphi_2 \leqslant \upvarphi_3$.
    \item $\upvarphi_1$ is minimal and $\upvarphi_3$ is maximal among the K\"{a}hler angles of $v$ with respect to all nonzero elements of $\mathcal{J}$.
\end{enumerate}   

In fact, $(J_1, J_2, J_3)$ is a basis for $\mathcal{J}$ that diagonalizes the symmetric bilinear form
$$
\mathcal{J} \times \mathcal{J} \to \Rl, \quad (J, J') \mapsto \cross{\pr_\mk{v} \circ J (v)}{\pr_\mk{v} \circ J' (v)}
$$
with $\cos^2(\upvarphi_i)||v||^2, \; i \in \set{1,2,3},$ on the diagonal. Such a triple $(\upvarphi_1, \upvarphi_2, \upvarphi_3)$ is called the quaternionic K\"{a}hler angle of $\mk{v}$ with respect to $v$.
\end{lemma}

The quaternionic K\"{a}hler angle of $\mk{v}$ in general depends on $v$. However, we have the following fact (see \cite[Lemma 2.4]{protohomogeneous} for a proof):

\begin{lemma}
If $\mk{v}$ is a protohomogeneous subspace of $\mb{H}^n$, then it has constant quaternionic K\"{a}hler angle, i.e. it does not depend on the choice of $v \in \mk{v} \mysetminus \set{0}$. In that case, we call $(\upvarphi_1, \upvarphi_2, \upvarphi_3)$ simply the quaternionic K\"{a}hler angle of $\mk{v}$.
\end{lemma}

The classification of protohomogeneous subspaces of $\mb{H}^n$ obtained in \cite[Theorem A]{protohomogeneous} shows that a $k$-dimensional protohomogeneous subspace $\mk{v} \subseteq \mb{H}^n$ is 'almost always' determined by its quaternionic K\"{a}hler angle up to $\Sp(n)\Sp(1)$: the only exception is when $k \leqslant n$ is congruent to $0$ or $3 \; (\mod 4)$, in which case there may be (at most) two noncongruent protohomogeneous subspaces of dimension $k$ with the same quaternionic K\"{a}hler angle. Since in our case $n=2$, this uncertainty does not concern us. From the table in the aforementioned classification theorem we see that there are the following possibilities for $\mk{v} \subseteq \mk{n}_2$:

\begin{itemize}
    \item \textit{Case 1:} $\dim \mk{v} = 2$ and the quaternionic K\"{a}hler angle of $\mk{v}$ is $(\upvarphi, \frac{\uppi}{2}, \frac{\uppi}{2})$ for some $\upvarphi \in [0, \frac{\uppi}{2}]$.
    \item \textit{Case 2:} $\dim \mk{v} = 3$ and the quaternionic K\"{a}hler angle of $\mk{v}$ is $(\upvarphi, \upvarphi, \frac{\uppi}{2})$ for some $\upvarphi \in \set{0, \frac{\uppi}{3}}$.
    \item \textit{Case 3:} $\dim \mk{v} = 4$ and the quaternionic K\"{a}hler angle of $\mk{v}$ is $(0, \upvarphi, \upvarphi)$ for some $\upvarphi \in [0, \frac{\uppi}{2}]$.
\end{itemize}

We start our investigation with case 1. First, assume that $\upvarphi = 0$. In that case, $\mk{v}$ is a totally complex subspace of $\mk{n}_2$ and we can take it to be $\Cx j E_{23}$ so that $\mk{n}_{2, \mk{v}} = \mk{g}_{\upalpha_1 + \upalpha_2} \oplus \Cx E_{23}$. Now we need to check whether $\mk{v}$ is admissible. To this end, we have to compute $N_{\mk{m}_2}(\mk{v})$ and see if its projection in $\mk{p}$ is the whole $\mk{b}_2$ or a proper subspace of it. Observe that
$$
\mk{m}_2 = \left\{ \left. \begin{pmatrix}
    p_{11} & p_{12} & 0 \\
    p_{21} & p_{22} & 0 \\
    0 & 0 & q \end{pmatrix} \right| \Re(p_{11}) + \Re(p_{22}) = 0, \Re(q) = 0 \right\} \simeq \mk{sl}(2, \Hq) \oplus \mk{sp}(1).
$$
By commuting such matrices with $jE_{23}$ and $kE_{23}$, one readily sees that
$$
N_{\mk{m}_2}(\mk{v}) = \left\{ \left. \begin{pmatrix}
    p_{11} & 0 & 0 \\
    p_{21} & p_{22} & 0 \\
    0 & 0 & q \end{pmatrix} \right| \Re(p_{11}) + \Re(p_{22}) = 0, p_{22} \in \Cx, q \in \Rl i \right\}.
$$
The projection of this in $\mk{p}$ is clearly equal to $\mk{b}_2$, so $\mk{v} = \Cx j E_{23} \subseteq \mk{n}_2$ produces a cohomogeneity-one action. The corresponding Lie algebra is
\begin{equation}\label{decomp1}
\mk{h}_{2, \mk{v}} = (\Rl i E_{11} \oplus (\Im \mb{H}) E_{22} \oplus \Rl i E_{33}) \oplus \mk{a} \oplus \mk{g}_{\upalpha_1} \oplus \Cx E_{23} \oplus \mk{g}_{\upalpha_1 + \upalpha_2},
\end{equation}
where the first summand in the parentheses is $\mk{h}_{2, \mk{v}} \cap \mk{k} = \mk{h}_{2, \mk{v}} \cap \mk{k}_0$. We claim that this action is orbit equivalent to the canonical extension of a cohomogeneity-one action on $B_1$ with a singular orbit of codimension 2. Indeed, in the notation of (5), we take $k=3$ and $\mk{w} = \mk{a}^1 \oplus \Cx E_{23}$. Then we have:
\begin{equation}\label{decomp2}
\mk{h}_{1, 3}^\Uplambda = N_{\mk{k}_1}(\mk{w}) \oplus \mk{a} \oplus \mk{g}_{\upalpha_1} \oplus \Cx E_{23} \oplus \mk{g}_{\upalpha_1 + \upalpha_2}.
\end{equation}
Observe that $\mk{h}_{2, \mk{v}}$ and $\mk{h}_{1, 3}^\Uplambda$ are Lie subalgebras of the parabolic subalgebra $\mk{q}_1$, and they both sit nicely within the Langlands decomposition $\mk{q}_1 = \mk{m}_1 \oplus \mk{a}_1 \oplus \mk{n}_1$: $\mk{h}_{2, \mk{v}} = (\mk{h}_{2, \mk{v}} \cap \mk{m}_1) \oplus (\mk{h}_{2, \mk{v}} \cap \mk{a}_1) \oplus (\mk{h}_{2, \mk{v}} \cap \mk{n}_1)$, and the same for $\mk{h}_{1, 3}^\Uplambda$. It follows that, in terms of the horospherical decomposition $M = B_1 \times A_1 \times N_1$, the singular orbits of the actions of $H_{2, \mk{v}}$ and $H_{1,3}^\Uplambda$ are $((H_{2, \mk{v}} \cap M_1) \ccdot o) \times A_1 \times N_1$ and $((H_{1,3}^\Uplambda \cap M_1) \ccdot o) \times A_1 \times N_1$, respectively. The first factors here are the singular orbits of the cohomogeneity-one actions of $H_{2, \mk{v}} \cap M_1$ and $H_{1,3}^\Uplambda \cap M_1$ on $B_1$, respectively. Looking at the decompositions (\ref{decomp1}) and (\ref{decomp2}), it is clear that these singular orbits coincide, since they both correspond to the Lie triple system $\mk{a}^1 \oplus \set{\uplambda E_{23} + \bar{\uplambda}E_{32} \mid \uplambda \in \Cx}$. Since the singular orbits of $H_{2, \mk{v}}$ and $H_{1,3}^\Uplambda$ coincide, these groups have the same orbits, for all the other orbits are just equidistant tubes around the singular one.

Next, assume that $\upvarphi = \frac{\uppi}{2}$. Then $\mk{v}$ is a totally real subspace of $\mk{n}_2$, so, acting by $K_2$ if necessary, we may assume $\mk{v} = \Rl E_{13} \oplus \Rl E_{23}$. Simple computations reveal that a matrix $X \in \mk{m}_2$ normalizing $\mk{v}$ must have $p_{12}, p_{21} \in \Rl$, which means that the image of the projection of $N_{\mk{m}_2}(\mk{v})$ to $\mk{p}$ will be a proper subspace of $\mk{b}_2$, so such $\mk{v}$ is not an admissible subspace.

We are left to consider the case $\upvarphi \in (0, \frac{\uppi}{2})$. Here $\mk{v}$ is a subspace of constant K\"{a}hler angle $\upvarphi$ inside a totally complex subspace of $\mk{n}_2$ of real dimension 4. Without loss of generality, we choose the latter to be $\Cx E_{13} \oplus \Cx E_{23}$. Then we can take $\mk{v}$ to be the real span of $E_{23}$ and $i \cos\upvarphi E_{23} + i \sin \upvarphi E_{13}$ (see \cite[Proposition 7]{berndt_bruck}). By commuting elements of $\mk{m}_2$ with these two vectors and solving simple systems of linear equations, one gets, among other things, that $\Re (p_{12}) = 0$ and $\Re (p_{21}) = 2 \Re (p_{11}) \cot \upvarphi$. It implies that the image of the projection of $N_{\mk{m}_2}(\mk{v})$ to $\mk{p}$ is contained in $\mk{b}_2 \cap \set{\Re(p_{11}) \cot\upvarphi - \Re(p_{12}) = 0}$, which is a linear hyperplane in $\mk{b}_2$. Therefore, this $\mk{v}$ is not admissible either.

In the second case, consider $\mk{v}$ with a quaternionic K\"{a}hler angle $(0, 0, \frac{\uppi}{2})$ first. Such $\mk{v}$ can be described as $(\Im \mb{H}) \hspace{0.5pt} v$ for some nonzero $v \in \mk{n}_2$. We may assume without loss of generality that $v = E_{23}$. One then easily computes:
$$
N_{\mk{m}_2}(\mk{v}) = \left\{ \left. \begin{pmatrix}
    p_{11} & 0 & 0 \\
    p_{21} & p_{22} & 0 \\
    0 & 0 & q \end{pmatrix} \right| \Re(p_{11}) + \Re(p_{22}) = 0, \hspace{1pt} q = \Im(p_{22}) \right\}.
$$
The image of the projection of that to $\mk{p}$ is the whole of $\mk{b}_2$, so this $\mk{v}$ is admissible. We have:
$$
\mk{h}_{2, \mk{v}} = ((\Im \mb{H}) E_{11} \oplus (\Im \mb{H})(E_{22} + E_{33})) \oplus \mk{a} \oplus \mk{g}_{\upalpha_1} \oplus \Rl E_{23} \oplus \mk{g}_{\upalpha_1 + \upalpha_2},
$$
where the first summand in the parentheses is $\mk{h}_{2, \mk{v}} \cap \mk{k} = \mk{h}_{2, \mk{v}} \cap \mk{k}_0 \simeq \mk{sp}(1) \oplus \mk{sp}(1)$. In a similar vein to what we did in case 1 with $\upvarphi = 0$, one can show that the orbits of $H_{2, \mk{v}}$ coincide with the orbits of $H_{1,2}^\Uplambda$ if we take $\mk{w} = \mk{a}^1 \oplus \Rl E_{23}$. Therefore, this $\mk{v}$ does not produce a new action.

Now let $\mk{v}$ be of quaternionic K\"{a}hler angle $(\frac{\uppi}{3}, \frac{\uppi}{3}, \frac{\uppi}{2})$. It follows from Proposition 5.3 and Remark 5.4 in \cite{protohomogeneous} that $\mk{v} = \vspan\set{E_{23}, iE_{23} + i\sqrt{3}E_{13}, jE_{23} - j\sqrt{3}E_{13}}$ does the trick. Simple calculations show that an element of $\mk{m}_2$ normalizing $\mk{v}$ must have $\Re (p_{11}) = \Re (p_{22}) = 0$, so the image of the projection of $N_{\mk{m}_2}(\mk{v})$ to $\mk{p}$ will be a proper subspace of $\mk{b}_2$, hence this $\mk{v}$ is not admissible.

Finally, assume $\mk{v}$ is 4-dimensional. First, let its quaternionic K\"{a}hler angle be $(0,0,0)$, that is, let it be a quaternionic line in $\mk{m}_2$. Without loss of generality, we choose $\mk{v} = \mb{H}E_{23} = \mk{g}_{\upalpha_2}$. One immediately sees that $N_{\mk{m}_2}(\mk{v}) = \mk{m}_2 \ominus \mk{g}_{\upalpha_1}$, whose projection in $\mk{p}$ is the whole $\mk{b}_2$, so this $\mk{v}$ is admissible. We have:
$$
\mk{h}_{2, \mk{v}} = \mk{g}_0 \oplus \mk{g}_{\upalpha_1} \oplus \mk{g}_{\upalpha_1 + \upalpha_2}.
$$
But this Lie subalgebra coincides with $\mk{h}_{1,1}^\Uplambda$ if we take $\mk{w} = \mk{a}^1$. Consequently, $\mk{v}$ does not give a new action.

Now let the quaternionic K\"{a}hler angle of $\mk{v}$ be $(0, \frac{\uppi}{2}, \frac{\uppi}{2})$. Then $\mk{v}$ is a totally complex subspace, hence we may assume $\mk{v} = \Cx E_{13} \oplus \Cx E_{23}$. Any element of $\mk{m}_2$ normalizing such a subspace must have $p_{12}, p_{21} \in \Cx$, so the image of the projection of $N_{\mk{m}_2}(\mk{v})$ to $\mk{p}$ is smaller than $\mk{b}_2$ and $\mk{v}$ is not admissible.

Finally, let $\mk{v}$ be of quaternionic K\"{a}hler angle $(0, \upvarphi, \upvarphi), \upvarphi \in (0, \frac{\uppi}{2})$. According to \cite{berndt_bruck} (see the discussion before Theorem 5 there), we can take $\mk{v}$ to be spanned by $E_{23}, iE_{23}, j\cos\upvarphi E_{23} + j\sin\upvarphi E_{13},$ and $k\cos\upvarphi E_{23} + k\sin\upvarphi E_{13}$. In a similar fashion to what we had in case 1, by solving the system of linear equations defining $N_{\mk{m}_2}(\mk{v})$, one gets -- among other things -- the same two linear dependencies $\Re (p_{12}) = 0$ and $\Re (p_{21}) = 2 \Re (p_{11}) \cot\upvarphi$. Therefore, the projection of $N_{\mk{m}_2}(\mk{v})$ in $\mk{p}$ is yet again contained in $\mk{b}_2 \cap \set{\Re (p_{11}) \cot\upvarphi - \Re (p_{12}) = 0}$ and $\mk{v}$ fails to be admissible. This finishes the proof of Theorem \ref{SL(3,H)/Sp(3)}.
\end{proof}

\section{Classification of cohomogeneity-one actions on \texorpdfstring{$\SO(5, \Cx)/\SO(5)$}{that space}}\label{second space}

In this section we classify, up to orbit equivalence, cohomogeneity-one actions on the noncompact dual of the compact Lie group $\Spin(5)$.

The symmetric space $\SO(5, \Cx)/\SO(5)$ is irreducible of rank 2 and dimension 10. Its root system is $B_2$ and can be identified with the root system of the complex simple Lie algebra $\mk{so}(5, \Cx)$. Consequently, all the root spaces have complex dimension 1. Let $\Uplambda = \set{\upalpha_1, \upalpha_2}$, where $\upalpha_1$ is the longest root. Then $\Upsigma^+ = \set{\upalpha_1, \upalpha_2, \upalpha_1 + \upalpha_2, \upalpha_1 + 2\upalpha_2}$. We also have $\mk{k}_0 = i\mk{a} = \Rl i H_{\upalpha_1} \oplus^\perp \Rl i H^2 = \Rl i H^1 \oplus^\perp \Rl i H_{\upalpha_2} \simeq \mk{u}(1) \oplus \mk{u}(1)$, where $i$ the complex structure of $\mk{so}(5,\Cx)$. The fact that $\upalpha_1$ and $\upalpha_2$ have different lengths implies that the corresponding boundary components $B_1 \simeq \Rl H^3$ and $B_2 \simeq \Rl H^3$ have different sectional curvatures ($- ||\upalpha_1||^2$ and $-||\upalpha_2||^2$, respectively) and thus are not congruent (which reflects the asymmetry of the Dynkin diagram $B_2$).

\begin{theorem}\label{SO(5,C)/SO(5)}
Let $H \subset G = \SO(5,\Cx)$ be a closed connected subgroup acting on $M$ with cohomogeneity one. Then its action is orbit equivalent to exactly one of the following:

\begin{enumerate}[\normalfont (1)]
    \item The action of the connected Lie subgroup $H_\ell$ of $G$ with Lie algebra
    $$
    \mk{h}_\ell = (\mk{a} \ominus \ell) \oplus \mk{n},
    $$
    where $\ell$ is a one-dimensional linear subspace of $\mk{a}$. Its orbits are all isometrically congruent to each other and form a Riemannian foliation on $M$.
    \item The action of the connected Lie subgroup $H_i, i \in \set{1,2},$ of $G$ with Lie algebra
    $$
    \mk{h}_i = \mk{a} \oplus (\mk{n} \ominus \ell_i),
    $$
    where $\ell_i$ is any one-dimensional linear subspace of $\mk{g}_{\upalpha_i}$. Its orbits form a Riemannian foliation on $M$ and there is exactly one minimal orbit.
    \item The action of the subgroup $\SO(4, \Cx)$ of $G$ embedded in a standard way. It has a $6$-dimensional totally geodesic singular orbit isometric to $\SO(4, \Cx)/\SO(4) \simeq \SL(2, \Cx)/\SU(2) \times \SL(2, \Cx)/\SU(2) \simeq \Rl H^3 \times \Rl H^3$.
    \item The action of the connected Lie subgroup $H^\Uplambda_{j,0}, j \in \set{1,2}$, of $G$ with Lie algebra
    $$
    \mk{h}_{j,0}^\Uplambda = \mk{k}_j \oplus \mk{a}_j \oplus \mk{n}_j.
    $$
    This action has a minimal singular orbit of codimension $3$ and can be obtained by canonical extension of the cohomogeneity-one action on $B_j$ with a single point as a singular orbit.
    \item The action of the connected Lie subgroup $H^\Uplambda_{j,1}, j \in \set{1,2}$, of $G$ with Lie algebra
    $$
    \mk{h}_{j,1}^\Uplambda = \mk{k}_0 \oplus \mk{a} \oplus \mk{n}_j,
    $$
    This action has a minimal singular orbit of codimension $2$ and can be obtained by canonical extension of the cohomogeneity-one action on $B_j$ with a geodesic as a singular orbit.
\end{enumerate}
\end{theorem}

\begin{proof}
We consider different cases of Theorem \ref{maintheorem}. If the orbits of $H$ form a foliation, we get the actions in (1) and (2). The action in (3) is the only one that has a totally geodesic singular orbit according to \cite{berndttamarusingtotgeod}. The stabilizer of $\SO(4,\Cx)$ at $o$ is $\SO(4)$, whose slice representation at $o$ is equivalent to the tautological one (we have $(\SO(4,\Cx) \ccdot o)^\perp \simeq \Rl H^4$), hence $\SO(4,\Cx)$ does indeed act with cohomogeneity one. It is also worth noting that we can take 
$$
\mk{so}(4, \Cx) =\mk{g}_0 \oplus \bigoplus_{\substack{\upalpha \in \Upsigma \\ \text{long root}}} \mk{g}_{\upalpha}.
$$
The long roots in $B_2$ form a root system $A_1 \sqcup A_1$ sitting inside $B_2$, and the hyperbolic spaces in the decomposition $\SO(4, \Cx)/\SO(4) \simeq \Rl H^3 \times \Rl H^3$ have the same curvature.

Now we determine the actions arising via the canonical extension method. Each boundary component $B_j$ is isometric to the real hyperbolic space $\Rl H^3$ (with different curvatures depending on $j$), which has precisely two cohomogeneity-one actions with a singular orbit up to strong orbit equivalence. One of them is the action of the (restricted) isotropy subgroup of $\Rl H^3$, which has a single point as a singular orbit and whose canonical extension is described in (4). The other one has a geodesic as a singular orbit and, by Proposition \ref{rank one proposition} and the observation made at the end of Subsection \ref{error correction}, is given by the connected Lie subgroup corresponding to $N_{\mk{k}_j}(\mk{a}^j) \oplus \mk{a}^j = \mk{k}_0 \oplus \mk{a}^j$. Its canonical extension is given in (5). Observe that the actions in (4) and (5) are not orbit equivalent to each other because the normal spaces of their singular orbits are tangent to the corresponding boundary components and thus have different sectional curvatures.

Now we proceed to the main part, namely to the nilpotent construction. Since the Dynkin diagram $B_2$ is asymmetric, we have to consider two cases.

\textit{Nilpotent construction with} $j=2$. In this case we have:
\begin{align*}
    \mk{n}_2^1 &= \mk{g}_{\upalpha_2} \oplus \mk{g}_{\upalpha_1 + \upalpha_2} \simeq \Cx^2, \\
    \mk{l}_2 &= \mk{g}_{-\upalpha_1} \oplus \mk{g}_0 \oplus \mk{g}_{\upalpha_1} = \mk{g}_2 \oplus \mk{z}_2 \oplus \mk{a}_2 \simeq \mk{sl}(2, \Cx) \oplus \mk{u}(1) \oplus \mk{u}(1) \simeq \mk{gl}(2, \Cx), \\
    \mk{k}_2 &= \mk{k}_0 \oplus \mk{k}_{\upalpha_1} = (\mk{g}_2 \cap \mk{k}) \oplus \mk{z}_2 \simeq \mk{su}(2) \oplus \mk{u}(1) \simeq \mk{u}(2).
\end{align*}
The adjoint representation of $\mk{g}_2$ on $\mk{n}_2^1$ is a nontrivial complex representation, hence it is equivalent to the irreducible representation of $\mk{sl}(2, \Cx)$ on $\Cx^2$. The adjoint representation of $\mk{a}_2 \oplus \mk{z}_2 = \Cx H^2$ on $\mk{n}_2^1$ is the standard complex representation by scalars because $H^2$ acts on $\mk{n}_2^1$ by multiplication by $\bilin{\upalpha_2}{H^2} = 1$. It follows that the adjoint representation of $\mk{l}_2$ (respectively, $\mk{k}_2$) on $\mk{n}_2^1$ is equivalent to the tautological representation of $\mk{gl}(2, \Cx)$ (respectively, $\mk{u}(2)$) on $\Cx^2$. Since $\mk{l}_2$ is a $\uptheta$-stable Lie subalgebra of $\mk{g}$, we may assume $\uptheta$ corresponds to the standard Cartan involution on $\mk{gl}(2, \Cx)$ (minus the adjoint of a matrix). Therefore, the problem of finding admissible and protohomogeneous subspaces of the \mbox{$L^0_2$\hspace{0.5pt}-module} $\mk{n}_2^1$ is equivalent to the analogous problem for the tautological representation of $\GL(2, \Cx)$.

Let $\mk{v} \subseteq \mk{n}_2^1$ be a linear subspace of dimension at least 2. According to Lemma 1 and Proposition 7 in \cite{berndt_bruck}, $\mk{v}$ is protohomogeneous if and only if it has constant K\"{a}hler angle. In particular, it must be even-dimensional. If $\mk{v} = \mk{n}_2^1$, then it is trivially admissible and we have $\mk{h}_{2, \mk{v}} = \mk{l}_2 \oplus \mk{g}_{\upalpha_1 + 2\upalpha_2} = \mk{so}(4, \Cx) \ominus \mk{g}_{-\upalpha_1 - 2\upalpha_2}$. Since we know that $H_{2, \mk{v}}$ acts with cohomogeneity one, it follows that it has the same orbits as $\SO(4, \Cx)$, so its action has already been taken into account in (3). Now suppose $\mk{v}$ has real dimension 2. It was shown in \cite[Theorem 6]{berdntdominguez-vazquez} that such a subspace is admissible if and only if its K\"{a}hler angle is zero, i.e. if it is a complex subspace. Up to the action of $K_2^0 \simeq \U(2)$, we may assume $\mk{v} = \mk{g}_{\upalpha_2}$. We have 
$$
\mk{h}_{2, \mk{v}} = \mk{l}_2 \oplus \mk{g}_{\upalpha_1 + \upalpha_2} \oplus \mk{g}_{\upalpha_1 + 2\upalpha_2} = \mk{g}_{-\upalpha_1} \oplus \mk{g}_0 \oplus \mk{n}_1.
$$
If we look back at (5), we see that $\mk{h}_{1,1}^\Uplambda$ is a Lie subalgebra of $\mk{h}_{2, \mk{v}}$. Since $H_{1,1}^\Uplambda$ acts with cohomogeneity one, its orbits coincide with the orbits of $H_{2, \mk{v}}$. Altogether, we see that the nilpotent construction method with $j=2$ does not give rise to any new actions.

\textit{Nilpotent construction with} $j=1$. In this case we have:
\begin{align*}
    \mk{n}_1^1 &= \mk{n}_1 = \mk{g}_{\upalpha_1} \oplus \mk{g}_{\upalpha_1 + \upalpha_2} \oplus \mk{g}_{\upalpha_1 + 2\upalpha_2} \simeq \Cx^3, \\
    \mk{l}_1 &= \mk{g}_{-\upalpha_2} \oplus \mk{g}_0 \oplus \mk{g}_{\upalpha_2} = \mk{g}_1 \oplus \mk{z}_1 \oplus \mk{a}_1 \simeq \mk{sl}(2, \Cx) \oplus \mk{u}(1) \oplus \mk{u}(1) \simeq \mk{gl}(2, \Cx), \\
    \mk{k}_2 &= \mk{k}_0 \oplus \mk{k}_{\upalpha_2} = (\mk{g}_1 \cap \mk{k}) \oplus \mk{z}_1 \simeq \mk{su}(2) \oplus \mk{u}(1) \simeq \mk{u}(2).
\end{align*}
We fix an isomorphism between $\mk{sl}(2, \Cx)$ and $\mk{g}_1$ by sending the standard basis $e,f,h$ of the former to $X, \uptheta X, h_{\upalpha_2} \in \mk{g}_1$, where $X \in \mk{g}_{\upalpha_2}, \uptheta X \in \mk{g}_{-\upalpha_2}$, and $h_{\upalpha_2} = \frac{2}{||\upalpha_2||^2} H_{\upalpha_2}$. Observe that $\mk{n}_1$ is an $\upalpha_2$-string, so it is an irreducible complex representation of $\mk{g}_1 \simeq \mk{sl}(2, \Cx)$. Since it is 3-dimensional, it is isomorphic to the adjoint representation of $\mk{sl}(2, \Cx)$. Now, $\mk{sl}(2, \Cx)$ is simple, hence it is also simple over $\Rl$, which means that the representation of $\mk{g}_1$ on $\mk{n}_1$ is irreducible as a real representation. An alternative description of this representation is the second symmetric power of the tautological representation of $\mk{sl}(2, \Cx)$. The adjoint representation of $\mk{a}_1 \oplus \mk{z}_1 = \Cx H^1$ on $\mk{n}_1^1$ is the standard complex representation by scalars because $H^1$ acts on $\mk{n}_1$ by multiplication by $\bilin{\upalpha_1}{H^1} = 1$. Consequently, we can describe the adjoint representation of $\mk{l}_1 \simeq \mk{sl}(2, \Cx) \oplus \Cx$ (respectively, $\mk{k}_1 \simeq \mk{su}(2) \oplus \mk{u}(1)$) on $\mk{n}_1$ as the exterior tensor product $\uprho_3 \otimes \upsigma$, where $\uprho_3$ is the irreducible 3-dimensional complex representation of $\mk{sl}(2, \Cx)$ (respectively, of $\mk{su}(2)$) and $\upsigma$ is the tautological representation of $\Cx$ (respectively, $\mk{u}(1)$) on $\Cx$.

Now we need to determine -- up to $K_1^0 \simeq \U(2)$ -- all subspaces $\mk{v}$ of $\mk{n}_1$ which are both protohomogeneous and admissible. We begin by borrowing an argument from \cite[Proposition 6]{berdntdominguez-vazquez}. The case $\mk{v} = \mk{n}_1$ can be excluded straight away since $\U(2)$ cannot act transitively on the 5-sphere. According to \cite[Proposition 5]{berdntdominguez-vazquez}, we may assume that $N_{\mk{m}_1}(\mk{v}) = \uptheta N_{\mk{m}_1}(\mk{n}_{1, \mk{v}})$ is contained in $(\mk{g}_1 \cap (\mk{g}_0 \oplus \mk{n})) \oplus \mk{z}_1 = \mk{k}_0 \oplus \mk{a}^1 \oplus \mk{g}_{\upalpha_2}$. In that case, $N_{\mk{k}_1}(\mk{v})$ is contained in $\mk{k}_0 \simeq \mk{u}(1) \oplus \mk{u}(1)$. But then $\mk{v}$ must be 2-dimensional, for $K_0^0 \simeq \U(1) \times \U(1)$ cannot act transitively on any sphere of dimension greater than one. We will show that -- under the assumption $N_{\mk{m}_1}(\mk{v}) \subseteq \mk{k}_0 \oplus \mk{a}^1 \oplus \mk{g}_{\upalpha_2}$ -- the only option for $\mk{v}$ is $\mk{g}_{\upalpha_1 + 2\upalpha_2}$.

Take nonzero vectors $e_j \in \mk{g}_{\upalpha_1 + j\upalpha_2}, j \in \set{0,1,2},$ such that $e_j = \ad(X) e_{j-1}$. Let $v = \sum_{j=0}^2 z_j e_j \in \mk{v}$ be a nonzero vector. A generic element of $\mk{k}_0 \oplus \mk{a}^1 \oplus \mk{g}_{\upalpha_2}$ can be written as
$$
Y = (a + ib) X + c h_{\upalpha_2} + i(d h_{\upalpha_2} + e H^1), 
$$
where $a, b, c, d,$ and $e$ are some real numbers. If $\mk{v}$ lies in $\mk{g}_{\upalpha_1} \oplus \mk{g}_{\upalpha_1 + \upalpha_2}$, then $a$ and $b$ must be zero for $Y$ to normalize $\mk{v}$, which means that the projection of $N_{\mk{m}_1}(\mk{v})$ in $\mk{p}$ will be smaller than $\mk{b}_1$ in that case. Hence, we may assume that $z_2 \ne 0$ and, for the sake of contradiction, that either $z_0$ or $z_1$ is also nonzero. Observe that $\ad(h_{\upalpha_2})$ acts diagonally on $\mk{n}_1 = \mk{g}_{\upalpha_1} \oplus \mk{g}_{\upalpha_1 + \upalpha_2} \oplus \mk{g}_{\upalpha_1 + 2\upalpha_2}$ with eigenvalues $-2, 0,$ and $2$, whereas $\ad(H^1)$ acts as the identity on the whole $\mk{n}_1$. Consequently, $i(d h_{\upalpha_2} + e H^1)$ acts on $\mk{n}_1$ diagonally with eigenvalues $i(e-2d), ie, i(e+2d)$. It implies that $N_{\mk{k}_1}(\mk{v}) \subseteq \mk{k}_0$ must be one-dimensional (otherwise it would be the whole $\mk{k}_0$ and $\mk{v}$ would have to be of dimension at least 3). In this situation, $N_{\mk{k}_1}(\mk{v})$ is spanned by some $i(d_0 h_{\upalpha_2} + e_0 H^1)$ and $\mk{v} = \Rl \hspace{0.2pt} v \hspace{0.5pt} \oplus \hspace{0.5pt} \Rl \hspace{1pt} \ad(i(d_0 h_{\upalpha_2} + e_0 H^1)) \hspace{0.4pt} v$. In order for $\ad(Y)$ to normalize $\mk{v}$, $\ad(Y) \hspace{0.5pt} v$ must be a linear combination of $v$ and $\ad(i(d_0 h_{\upalpha_2} + e_0 H^1)) v$, which boils down to the following system of equations:
$$
\begin{cases}
(-2c + i(e-2d))z_1 = \uplambda z_1 + i\upmu (e_0 - 2d_0)z_1, \\
(a+ib)z_1 + i e z_2 = \uplambda z_2 + i\upmu e_0 z_2, \\
(a+ib)z_2 + (2c + i(e+2d))z_3 = \uplambda z_3 + i\upmu (e_0 + 2d_0)z_3.
\end{cases}
$$
First assume that $z_1 \ne 0$. The first equation then implies $\uplambda = -2c$, which, when substituted into the second one, gives 
$$
a+ib = \frac{z_2}{z_1}(-2c + i(\upmu e_0 - e)).
$$
By writing $\dfrac{z_2}{z_1} = f + ig$, we obtain
$$
\begin{cases}
a = -2fc - g(\upmu e_0 - e), \\
b = -2gc + f(\upmu e_0 - e).
\end{cases}
$$
Whatever $f$ and $g$ are, we get a linear dependency on $a, b, e$, which means that $\mk{v}$ cannot be admissible. In case $z_1 = 0$ but $z_2 \ne 0$, the second equation of our system implies $\uplambda = 0$, which transforms the third equation into
$$
a+ib = \frac{z_3}{z_2}(-2c + i(\upmu(e_0 + 2d_0) - e - 2d)).
$$
In a similar way, we get a linear dependency on $a, b, e$.

The upshot of the above argument is that, up to $K_1^0$, $\mk{v}$ must be equal to $\mk{g}_{\upalpha_1 + 2\upalpha_2}$. Observe that this is indeed an admissible and protohomogeneous subspace. Since the representation of $\SU(2) \subset K_1^0$ on $\mk{n}_1$ is equivalent to the 3-dimensional complex irreducible representation of $\SU(2)$, there exists an element of $K_1^0$ that maps $\mk{g}_{\upalpha_1 + 2\upalpha_2}$ onto $\mk{g}_{\upalpha_1}$ (in terms of the representation on the space of quadratic polynomials, we can take the special unitary matrix $\left[ \begin{smallmatrix} 0 & -1 \\ 1 & \hspace{0.5em}0 \end{smallmatrix} \right]$, which induces $y^2 \mapsto x^2$). We have:
\begin{align*}
    \mk{n}_{1, \mk{g}_{\upalpha_1}} &= \mk{g}_{\upalpha_1 + \upalpha_2} \oplus \mk{g}_{\upalpha_1 + 2\upalpha_2}, \\
    N_{\mk{m}_1}(\mk{n}_{1, \mk{g}_{\upalpha_1}}) &= \uptheta N_{\mk{m}_1}(\mk{g}_{\upalpha_1}) = \mk{k}_0 \oplus \mk{a}^1 \oplus \mk{g}_{\upalpha_2}, \\
    \mk{h}_{1, \mk{g}_{\upalpha_1}} &= N_{\mk{m}_1}(\mk{n}_{1, \mk{g}_{\upalpha_1}}) \oplus \mk{a}_1 \oplus \mk{n}_{1, \mk{g}_{\upalpha_1}} = \mk{g}_0 \oplus \mk{g}_{\upalpha_2} \oplus \mk{g}_{\upalpha_1 + \upalpha_2} \oplus \mk{g}_{\upalpha_1 + 2\upalpha_2} = \mk{g}_0 \oplus \mk{n}_2.
\end{align*}
But looking at the actions given in (5), one sees that $\mk{h}_{1, \mk{g}_{\upalpha_1}} \hspace{-2.5pt} = \hspace{0.5pt} \mk{h}_{2,1}^\Uplambda$, which means that the actions of $H_{2,1}^\Uplambda$ and $H_{1, \mk{g}_{\upalpha_1}}$ have the same orbits. We conclude that the nilpotent construction method does not give any new actions for $M$.
\end{proof}

\section{Classification of cohomogeneity-one actions on the noncompact complex two-plane Grassmannians}\label{third space}

In this section we classify, up to orbit equivalence, cohomogeneity-one actions on the noncompact complex Grassmann manifolds $\mathrm{Gr}^*(2, \Cx^{n+4}) = \SU(n+2,2)/\mathrm{S}(\U(n+2)\U(2)), n \geqslant 1$.

The symmetric space $M = \SU(n+2,2)/\mathrm{S}(\U(n+2)\U(2))$ is irreducible of rank 2 and dimension $4n+8$, and its root system is $BC_2$. This is the complex analog of the symmetric space $\mathrm{Gr}^*(2, \Rl^{n+4}) = \SO^0(n+2,2)/\SO(n+2)\SO(2)$ of type $B_2$, the classification of cohomogeneity-one actions on which, as we mentioned earlier, was obtained by Berndt and Dom\'{i}nguez-V\'{a}zquez in \cite{berdntdominguez-vazquez}. If we choose simple roots $\upalpha_1$ and $\upalpha_2$ such that $2\upalpha_2$ is also a root, then we have $\Upsigma^+ = \set{\upalpha_1, \upalpha_2, 2\upalpha_2, \upalpha_1 + \upalpha_2, \upalpha_1 + 2\upalpha_2, 2\upalpha_1 + 2\upalpha_2}$, where $\upalpha_1$ and $\upalpha_1 + 2\upalpha_2$ have multiplicity 2, $\upalpha_2$ and $\upalpha_1 + \upalpha_2$ have multiplicity $2n$, and $2\upalpha_2$ and $2\upalpha_1 + 2\upalpha_2$ have multiplicity 1. The Lie algebra $\mk{k}_0$ is isomorphic to $\mk{u}(n) \oplus \mk{u}(1)$. Observe that $M$ is a Hermitian symmetric space and it also has a quaternion-K\"{a}hler structure. Let $I$ and $\mathcal{J} \subseteq \End(TM)$ denote the almost complex and quaternion-K\"{a}hler structures of $M$, respectively (the latter is a rank-3 subbundle of the endomorphism bundle of $TM$). We will use the interplay between them to distinguish between non-equivalent cohomogeneity-one actions on $M$. The boundary components $B_1$ and $B_2$ are isometric to $\Cx H^{n+1}$ and $\Rl H^3$, respectively. Therefore, for the first time, we are encountering a noncompact rank-2 symmetric space containing a boundary component not isometric to a real hyperbolic space. The reason why this is special is because, unlike real hyperbolic spaces, complex hyperbolic spaces have a nondiscrete moduli space of cohomogeneity-one actions with a singular orbit, so the canonical extension method, when applied to $B_1$, will produce a one-parameter family of cohomogeneity-one actions with a singular orbit on $M$. In order to be able to formulate the theorem, we need to know what the almost complex structure $I$ on $M$ looks like in terms of the restricted root space decomposition of $\mk{g}$. With respect to $I_o$, $\mk{p}_{\upalpha_2}$ and $\mk{p}_{\upalpha_1 + \upalpha_2}$ are complex subspaces of $\mk{p}$. Moreover, $I_o$ swaps $\mk{p}_{\upalpha_1}$ and $\mk{p}_{\upalpha_1 + 2\upalpha_2}$, so each of them is a totally real subspace. Finally, $I_o$ sends $\Rl H_{\upalpha_2}$ to $\mk{p}_{2\upalpha_2}$ and $\Rl H^1$ to $\mk{p}_{2\upalpha_1 + 2\upalpha_2}$. We pull this complex structure back to $\mk{a} \oplus \mk{n}$ along the isomorphism $\mk{a} \oplus \mk{n} \isoto \mk{p}$. Note that $B_1$ is a complex submanifold of $M$, whereas $B_2$ is a totally real one.

\begin{theorem}\label{Gr^*(2, Cx^{n+2})}
Let $H \subset G = \SU(n+2,2)$ be a closed connected subgroup acting on $M$ with cohomogeneity one. Then its action is orbit equivalent to exactly one of the following:

\begin{enumerate}[\normalfont (1)]
    \item The action of the connected Lie subgroup $H_\ell$ of $G$ with Lie algebra
    $$
    \mk{h}_\ell = (\mk{a} \ominus \ell) \oplus \mk{n},
    $$
    where $\ell$ is a one-dimensional linear subspace of $\mk{a}$. Its orbits are all isometrically congruent to each other and form a Riemannian foliation on $M$.
    \item The action of the connected Lie subgroup $H_i, i \in \set{1,2},$ of $G$ with Lie algebra
    $$
    \mk{h}_i = \mk{a} \oplus (\mk{n} \ominus \ell_i),
    $$
    where $\ell_i$ is any one-dimensional linear subspace of $\mk{g}_{\upalpha_i}$. Its orbits form a Riemannian foliation on $M$ and there is exactly one minimal orbit.
    \item The action of the subgroup $\SU(n+1,2)$ of $G$ embedded in a standard way. It has a totally geodesic singular orbit of codimension $4$ isometric to $\mathrm{Gr}^*(2, \Cx^{n+3}) = \SU(n+1,2)/\mathrm{S}(\U(n+1)\U(2))$. This orbit is a complex/quaternionic submanifold with respect to the complex/quaternion-K\"{a}hler structures of $M$.
    \item The action of the subgroup $\SU(n+2,1)$ of $G$ embedded in a standard way. It has a totally geodesic singular orbit of dimension $2n+4$ isometric to $\mathrm{Gr}^*(1, \Cx^{n+3}) = \Cx H^{n+2} = \SU(n+2,1)/\mathrm{S}(\U(n+2)\U(1))$. This orbit is a complex/totally complex submanifold with respect to the complex/quaternion-K\"{a}hler structures of $M$.
    \item In case $n = 2m \geqslant 2$, the action of the subgroup $\Sp(m+1,1)$ of $G$ embedded in a standard way. It has a totally geodesic singular orbit of dimension $2n + 4$ isometric to $\mathrm{Gr}^*(1, \Hq^{m+2}) = \Hq H^{m+1} = \Sp(m+1,1)/\Sp(m+1)\Sp(1)$. This orbit is a totally real/quaternionic submanifold with respect to the complex/quaternion-K\"{a}hler structures of $M$.
    \item The action of the connected Lie subgroup $H^\Uplambda_{2,k}, k \in \set{0,1}$, of $G$ with Lie algebra
    $$
    \mk{h}_{2,0}^\Uplambda = \mk{k}_2 \oplus \mk{a}_2 \oplus \mk{n}_2, \quad \mk{h}_{2,1}^\Uplambda = \mk{k}_0 \oplus \mk{a} \oplus \mk{n}_2.
    $$
    This action has a minimal singular orbit of codimension $3-k$ and can be obtained by canonical extension of the cohomogeneity-one action on $B_2 \simeq \Rl H^3$ with a single point {\normalfont (}$k=0${\normalfont )} or geodesic {\normalfont (}$k=1${\normalfont )} as a singular orbit.
    \item The action of the connected Lie subgroup $H_{1, (\upvarphi, k)}^\Uplambda$, $\upvarphi \in [0, \frac{\uppi}{2}]$, of $G$ with Lie algebra
    $$
    \mk{h}_{1, (\upvarphi, k)}^\Uplambda = N_{\mathfrak{k}_1}(\mathfrak{w}) \oplus \mk{w} \oplus \mk{a}_1 \oplus \mk{n}_1,
    $$
    where $\mk{w} \subseteq \mk{a}^1 \oplus \mk{g}_{\upalpha_2} \oplus \mk{g}_{2\upalpha_2}$ is the orthogonal complement of a subspace $\mk{w}^\perp \subseteq \mk{a}^1 \oplus \mk{g}_{\upalpha_2} \oplus \mk{g}_{2\upalpha_2}$ of dimension $k$ and constant K\"{a}hler angle $\upvarphi$, and
    
    \begin{enumerate}[\normalfont (a)]
        \item For $\upvarphi = 0$, $k \in \set{2, 4, \ldots, 2n+2}$, and $\mk{w}^\perp$ is required to lie in $\mk{g}_{\upalpha_2}$ unless $k = 2n+2$. In this case, $\mk{w}^\perp$ is just a complex subspace of complex dimension $k/2$;
        \item For $\upvarphi = \frac{\uppi}{2}$, $k \in \set{2, 3, \ldots, n+1}$, and $\mk{w}^\perp$ is required to lie in $\mk{g}_{\upalpha_2}$ unless $k = n+1$, in which case it is required to contain $\mk{g}_{2\upalpha_2}$. In this case, $\mk{w}^\perp$ is a totally real subspace;
        \item For $\upvarphi \in (0, \frac{\uppi}{2})$, $k \in \set{2, 4, \ldots, 2\lfloor \frac{n}{2} \rfloor}$, and $\mk{w}^\perp$ is required to lie in $\mk{g}_{\upalpha_2}$.
    \end{enumerate}
    
    This action has a minimal singular orbit of codimension $k$ and can be obtained by canonical extension of a cohomogeneity-one action on $B_1 \simeq \Cx H^{n+1}$ with a singular orbit.
\end{enumerate}
\end{theorem}

\begin{proof}
We consider different cases of Theorem \ref{maintheorem}. If the orbits of $H$ form a foliation, we get the actions in (1) and (2).

The actions in (3), (4), and (5) are the only ones with a totally geodesic singular orbit according to \cite{berndttamarusingtotgeod}. Let us describe these three orbits in a bit more detail. Recall that $M$ can be thought of as the set of $2$-dimensional complex subspaces of $\Cx^{n+4}$ on which the restriction of the standard indefinite Hermitian form of signature $(n+2,2)$ is negative definite. The group $G = \SU(n+2,2)$ acts transitively on that set and the stabilizer of $o = \langle e_{n+3}, e_{n+4} \rangle_\Cx$ is precisely $K = \Ss(\U(n+2)\U(2))$.

Consider the subset $S_3$ of $M$ consisting of those $2$-dimensional subspaces that lie in the complex hyperplane $\langle e_2, \ldots, e_{n+4} \rangle_\Cx$. This subset is $\mathrm{Gr}^*(2,\Cx^{n+3})$ embedded totally geodesically into $M$. The subgroup $\SU(n+1,2) \subset G$ of elements fixing the first basis vector $e_1$ preserves $S_3$ and acts transitively on it. Its isotropy subgroup at $o$ is $\Ss(\U(n+1)\U(2))$, whose subgroup $\U(2)$ acts transitively on the unit sphere in $N_o S_3$, so $\SU(n+1,2)$ does indeed act on $M$ with cohomogeneity one and has $S_3$ as its singular orbit (we have $S_3^\perp \simeq \Cx H^2$). 

Similarly, define $S_4$ to be the subset of $M$ consisting of those $2$-dimensional subspaces of $\Cx^{n+4}$ that contain the last basis vector $e_{n+4}$. These are in bijective correspondence with the complex lines in the hyperplane $\langle e_1, \ldots, e_{n+3} \rangle_\Cx$ on which the restriction of the standard indefinite Hermitian form of signature $(n+2,1)$ is negative definite, so $S_4$ is actually $\Gr^*(1, \Cx^{n+3}) = \Cx H^{n+2}$ embedded totally geodesically into $M$. The subgroup $\SU(n+2,1) \subset G$ of elements fixing $e_{n+4}$ preserves $S_4$ and acts transitively on it. Its isotropy subgroup at $o$ is $\Ss(\U(n+2)\U(1))$ and its slice representation at $o$ is equivalent to the tautological representation $\U(n+2) \curvearrowright \Cx^{n+2}$, also known as the isotropy representation of $\Cx H^{n+2}$, which is a reflection of the fact that $S_4^\perp \simeq \Cx H^{n+2}$. It is not hard to show that $S_4^\perp$ is a reflective submanifold of $M$. But now observe that $(S_4^\perp)^\perp = S_4$ is of rank 1, so, according to Subsection \ref{classes_of_actions}, $S_4^\perp$ is also a totally geodesic singular orbit of some cohomogeneity-one action. We claim that this action is orbit equivalent to the action of $\SU(n+2,1)$, i.e. that $S_4$ and $S_4^\perp$ are congruent. Indeed, $S_4$ is easily seen to be a totally complex submanifold with respect to the quaternion-K\"{a}hler structure of $M$, which means that there exists a one-dimensional subspace $\mathcal{J}'_o \subset \mathcal{J}_o$ preserving $T_o S_4$ and $J(T_o S_4) = N_o S_4 = T_o S_4^\perp$ for each $J \in \mathcal{J}_o$ orthogonal to $\mathcal{J}'_o$. Each such $J$ of square $-1$ can be realized as the differential at $o$ of some $k \in K$ (see the discussion of the quaternion-K\"{a}hler structure of $M$ below), so $k$ is an isometry mapping $S_4$ onto $S_4^\perp$. Note that the geodesic reflection in $S_4$ is a holomorphic involutive isometry. The submanifold $S_4$ is called a complex form of $M$ (see \cite{wolf_complexforms}). 

Finally, if $n=2m \geqslant 2$, consider the standard indefinite quaternionic Hermitian form $H$ on $\Hq^{m+2}$ of signature $(m+1,1)$ and identify $\Hq^{m+2}$ with $\Cx^{2m+4}$ in a standard way. Note that if we write $H = h - \omega j$, where $h, \omega \colon \Hq^{m+2} \times \Hq^{m+2} \to \Cx$, then $h$ is precisely our indefinite Hermitian form on $\Cx^{n+4}$. Consider the subset $S_5$ of $M$ consisting of the quaternionic lines in $\Cx^{n+4}$ the restriction of $H$ to which is negative definite. This subset is $\Gr^*(1, \Hq^{m+2}) = \Hq H^{m+1}$ embedded totally geodesically into $M$. The group $\Sp(m+1,1)$ sits naturally inside $G$ and preserves $S_5$. Its isotropy subgroup at $o$ is $\Sp(m+1)\Sp(1)$, whose slice representation at $o$ is equivalent to the standard representation on $\Hq^{m+1}$. Just like before, we have $S_5^\perp \simeq \Hq H^{m+1} \simeq S_5$, and these two submanifolds are in fact congruent. Their tangent spaces at $o$ are mapped to each other by $I_o$ because $S_5$ is totally real. But just as was the case with the almost complex structures in $\mathcal{J}_o$, $I_o$ also comes from $K$, so the corresponding $k \in K$ maps $S_5$ onto $S_5^\perp$. The geodesic reflection in $S_5$ is an antiholomorphic involutive isometry, for which reason $S_5$ is called a real form of $M$ (see \cite{jaffee_realforms}).

Now we proceed to the canonical extension method. There are two strong orbit equivalence classes of cohomogeneity-one actions on $B_2 \simeq \Rl H^3$ with a singular orbit, whose extensions are described in (6). Now consider $B_1 \simeq \Cx H^{n+1}$. All cohomogeneity-one actions on $B_1$ with a non-totally-geodesic singular orbit come from the nilpotent construction method. The nilpotent construction representation for $\Cx H^{n+1}$ is equivalent to the tautological representation $\U(n) \curvearrowright \Cx^n$. We know that if two protohomogeneous subspaces of $\mk{g}_{\upalpha_2} \simeq \Cx^n$ are $\U(n)$-congruent, they produce strongly orbit equivalent cohomogeneity-one actions on $\Cx H^{n+1}$. It was shown in \cite[Theorem 4.1(ii)]{berndttamarurankone} that the converse is true. To be precise, the following are equivalent for protohomogeneous subspaces $\mk{v}, \mk{v}' \subseteq \mk{g}_{\upalpha_2}$:

\begin{enumerate}[(i)]
    \item The cohomogeneity-one actions on $\Cx H^{n+1}$ produced by $\mk{v}$ and $\mk{v}'$ are orbit equivalent.
    \item The cohomogeneity-one actions on $\Cx H^{n+1}$ produced by $\mk{v}$ and $\mk{v}'$ are strongly orbit equivalent.
    \item $\mk{v}$ and $\mk{v}'$ are $\U(n)$-congruent.
\end{enumerate}

So what we really need is the classification of protohomogeneous subspaces of $\Cx^n$ with respect to $\U(n)$. This was carried out by Berndt and Br\"{u}ck in \cite{berndt_bruck}. They showed that a subspace of $\Cx^n$ is protohomogeneous if and only if it has constant K\"{a}hler angle, and two such subspaces are $\U(n)$-congruent if and only if they have the same dimension and K\"{a}hler angle. Together with the classification of cohomogeneity-one actions on $\Cx H^{n+1}$ with a totally geodesic singular orbit (see \cite[Theorem 1]{berndt_bruck}), Proposition \ref{rank one proposition} and the observation at the end of Subsection \ref{error correction} imply that actions in (7) exhaust the list of actions coming from those on $B_1$ by canonical extension. No two actions from (7) are mutually orbit equivalent by design: the normal spaces of their singular orbits differ either in dimension or (constant) K\"{a}hler angle. Here we use the fact that every isometry of $M$ is either holomorphic or anti-holomorphic so it preserves the K\"{a}hler angles of tangent subspaces (holomorphic and anti-holomorphic isometries constitute the two connected components of $I(M)$).

There is one action in (6) that may, in theory, have a totally geodesic singular orbit, namely the action of $H_{1,(0,4)}^\Uplambda$. Its orbit $S = H_{1,(0,4)}^\Uplambda \ccdot \hspace{0.8pt} o$ is a complex submanifold of $M$ of codimension 4. But so is the singular orbit of the action of $\SU(n+1,2) \subset G$ given in (3), so the action of $H_{1,(0,4)}^\Uplambda$ may be orbit equivalent to that of $\SU(n+1,2)$. To show that this is not the case, we have to prove that $S$ is not totally geodesic. To this end, we need to show the second fundamental form $\II$ of $S$ is nonzero. Note that 
$$
T_oS = \pr_\mk{p}(\mk{w}) \oplus \mk{a}_1 \oplus \mk{n}_1, \quad N_oS = \pr_\mk{p}(\mk{w}^\perp).
$$
Recall that $\mk{w}^\perp$ is contained in $\mk{g}_{\upalpha_2}$ if $n>1$ and coincides with $\mk{a}^1 \oplus \mk{g}_{\upalpha_2} \oplus \mk{g}_{2\upalpha_2}$ if $n=1$. Since $[\mk{g}_{\upalpha_1 + \upalpha_2}, \mk{g}_{-\upalpha_1}] = \mk{g}_{\upalpha_2}$, we can find $X \in \mk{g}_{\upalpha_1 + \upalpha_2}$ and $Y \in \mk{g}_{-\upalpha_1}$ such that $[X,Y] \in \mk{g}_{\upalpha_2} \mysetminus \mk{w}$. Note that $X - \uptheta X$ and $Y - \uptheta Y$ lie in $T_oS$. Since $S$ is an orbit of $H_{1,(0,4)}^\Uplambda$ and the Riemannian metric on $M$ comes from the Killing form of $\mk{g}$, we have $\II(X - \uptheta X, Y - \uptheta Y) = \pr_{N_oS}([Z, Y - \uptheta Y])$, where $Z \in \mk{k}$ is any vector such that $X - \uptheta X + Z \in \mk{h}_{1,(0,4)}^\Uplambda$ (see, for example\footnote{Any other method of computing the second fundamental form will work equally well here, this is just one way to do it.}, \cite{IIformula}). We pick $Z = X + \uptheta X$, which gives
$$
    \II(X - \uptheta X, Y - \uptheta Y) = \pr_{N_oS}([X + \uptheta X, Y - \uptheta Y]) = \pr_{N_oS}([X, Y] - \uptheta [X,Y]) \ne 0,
$$
which implies that the fundamental form of $S$ is nonzero at $o$ and $S$ is not totally geodesic.

Now we need to prove that the actions in (6) and (7) are mutually nonequivalent. We are going to show that, given an action from (6) with a singular orbit $S$ and one from (7) with a singular orbit $S'$ of the same dimension as $S$, normal spaces to $S$ and $S'$ have either different K\"{a}hler angles or vectors of different holomorphic sectional curvatures. Computing all such sectional curvatures by hand would be a daunting task, so we will employ a different strategy and leverage the abundance of geometric structures on $M$. In \cite{berndtgrassmannian}, Berndt studied the complex Grassmannian of two-planes $\mathrm{Gr}(2, \Cx^{n+4})$, which is the compact dual of our symmetric space $M$, so we denote it by $M^*$. Since holonomy is preserved under duality, $M^*$ is a Hermitian symmetric space and it also has a quaternion-K\"{a}hler structure. Write $I^*$ for the almost complex structure of $M^*$ and $\mathcal{J}^*$ for the 3-dimensional subbundle of $\End(T M^*)$ given by the quaternion-K\"{a}hler structure of $M^*$. Berndt showed that holomorphic sectional curvatures of $M^*$ lie between $4$ and $8$ after a suitable rescaling of the Riemannian metric. Moreover, he proved that, given any $p \in M^*$ and a nonzero $X \in T_p M^*$, the following are equivalent:

\begin{enumerate}[(i)]
    \item $X$ is singular.
    \item $I_p^* X \perp \mathcal{J}^*_p X$ or $I_p^* X \in \mathcal{J}^*_p X$.
    \item The holomorphic sectional curvature determined by $X$ is either minimal or maximal (that is, either $4$ or $8$).
\end{enumerate}

Recall that $X$ is called singular if it is tangent to more than one maximal flat in $M^*$. It is equivalent to asking that $X$ is a singular point with respect to the restricted isotropy representation of $M^*$, i.e. that the $\Ss(\U(n+2)\U(2))$-orbit of $X$ in $T_p M^*$ is singular. From this second description one sees that the set of regular (= nonsingular) vectors in $T_p M^*$ is connected, open, and dense. In (ii) and (iii), $I^*_p X \perp \mathcal{J}^*_p X$ corresponds to $K_\mathrm{hol}(X) = 4$, while $I_p^* X \in \mathcal{J}_p^*X$ corresponds to $K_\mathrm{hol}(X) = 8$. So this result allows us to divide singular tangent vectors to $M^*$ into two groups depending on their holomorphic sectional curvature. We temporarily call a singular vector $X$ \textbf{of type A} if $I_p^* X \perp \mathcal{J}_p^* X$ and \textbf{of type B} if $I_p^* X \in \mathcal{J}_p^* X$. Since duality preserves holonomy and isotropy subgroups and changes the sign of sectional curvatures, the whole discussion applies to $M$ with 'minimal' interchanged with 'maximal' in (iii) and $4$ and $8$ replaced with $-4c$ and $-8c$, respectively (our normalization of the Riemannian metric may differ from the one in \cite{berndtgrassmannian}). Now, the fact that is of paramount importance for us and that allows to use this interplay between the two structures of $M$ is that isometries preserve not only the set of singular vectors but their type. Indeed, any isometry is either holomorphic or anti-holomorphic, so it preserves holomorphic sectional curvatures and hence the type of a singular vector.

We already know what the complex structure of $M$ looks like with respect to the restricted root space decomposition, so now we need to describe the quaternion-K\"{a}hler structure. We have $\mk{g} = \mk{su}(n+2,2)$, and its Cartan decomposition is:
$$
\mk{k} = \left\{ \left. \left[
        \begin{array}{c | c}
            \begin{array}{c c c}
                & & \\
                & \mk{u}(n+2) & \\
                &  &
            \end{array} & 0 \\
            \hline
            0 & \mk{u}(2)
        \end{array}
    \right] \right| \tr = 0 \right\}, \quad 
    \mk{p} = \left\{ \left[
        \begin{array}{c | c}
            \begin{array}{c c c}
                & & \\
                & 0 & \\
                &  &
            \end{array} & B \\
            \hline
            B^* & 0
        \end{array}
    \right] \right\},
$$
where $B$ runs through $\mathrm{Mat}((n+2) \times 2, \Cx)$. The restricted isotropy representation of $M$ is just the adjoint representation of $K = \Ss(\U(n+2)\U(2))$ on $\mk{p}$. Note that if we identify $\mk{p}$ with $\Cx^{n+2} \oplus \Cx^{n+2} \cong \Cx^{n+2} \otimes \Cx^2$ in the obvious way, then the restricted isotropy representation is isomorphic to the restriction of the external tensor product representation $\uprho_{n+2} \otimes \uprho_2$ of $\U(n+2) \times \U(2)$ to $K$ (here $\uprho_k$ is the tautological representation of $\U(k)$ on $\Cx^k$). We have a normal subgroup $\Sp(1) \simeq \SU(2) \trianglelefteq K$, and its representation on $\mk{p}$ is just $1 \otimes \uprho_2$, so it extends to a unital $\Rl$-algebra homomorphism $\Hq \to \End(\mk{p})$, i.e. a quaternionic structure on $\mk{p}$. The fact that $\SU(2)$ is a normal subgroup of $K = \mathrm{Hol}(M,o)$ means that this structure is holonomy-invariant and thus gives a parallel quaternionic structure on $M$, while $\SU(2) \subseteq \O(\mk{p})$ means that this structure is compatible with the Riemannian metric and therefore makes $M$ into a quaternion-K\"{a}hler manifold. Note that -- by construction -- every almost complex structure $J$ in $\mathcal{J}_o$ lies in the image of the isotropy representation, i.e. it is the differential at $o$ of some isometry $k \in K$. We have already used this fact before when we described the totally geodesic singular orbits of the actions in (3)-(5).

Let us take $\mk{a} = \Rl(E_{n+2,n+3} + E_{n+3,n+2}) \oplus \Rl(E_{n+1,n+4} + E_{n+4,n+1})$. The restricted root space decomposition of $\mk{g}$ with respect to such a choice of $\mk{a}$ is explicitly described in Knapp (\cite[p. 371, Example 2]{knapp}). Using that, as well as our description of the quaternion-K\"{a}hler structure, one easily sees how $\mathcal{J}_o \subseteq \End(\mk{p})$ acts on different root vectors. Note also that the almost complex structure on $M$ is given at the point $o$ simply by
$$
I_o \colon \; \left[ \begin{array}{c | c}
            0 & B \\
            \hline
            B^* & 0
        \end{array} \right] \mapsto \left[ \begin{array}{c | c}
            0 & iB \\
            \hline
           -i B^* & 0
        \end{array} \right].
$$
Finally, observe that $\mk{su}(n+2,2)$ is a real form of $\mk{sl}(n+4, \Cx)$, so the Killing form $B$ of $\mk{g}$ is just the restriction of the Killing form of $\mk{sl}(n+4, \Cx)$, which is $(2n+8)\tr$. Now we are ready to search for singular vectors of types $A$ and $B$ in $\mk{p}$. For instance, one can compute that $\mathcal{J}_o$ sends $\mk{p}_{2\upalpha_2}$ onto $\Rl H_{\upalpha_2} \oplus \Cx (E_{n+1,n+3} + E_{n+3,n+1})$, while $I_o$ sends it onto $\Rl H_{\upalpha_2}$, so we deduce that $\mk{p}_{2\upalpha_2}$ consists entirely of singular vectors of type $B$. In a similar fashion, one calculates that

\begin{itemize}
    \item $\mk{p}_{\upalpha_1}$ and $\mk{p}_{\upalpha_1 + 2\upalpha_2}$ consist of singular vectors of type $A$, while
    \item $\mk{p}_{2\upalpha_1 + 2\upalpha_2}, \mk{p}_{\upalpha_2}, \mk{p}_{\upalpha_1 + \upalpha_2}, \Rl H^1,$ and $\Rl H_{\upalpha_2}$ all consist of singular vectors of type $B$.
\end{itemize}

Consider the action of $H_{2,1}^\Uplambda$. The normal space of its singular orbit at $o$ is $\mk{p}_{\upalpha_1}$, which is a totally real subspace of $\mk{p}$. The only action in (7) the normal spaces of whose singular orbit are totally real of dimension 2 is the action of $H_{1,(\pi/2,2)}^\Uplambda$. Whatever $n$ is, $\mk{w}^\perp$ intersects $\mk{g}_{\upalpha_2}$ and thus $N_o(H_{1,(\pi/2,2)}^\Uplambda \ccdot o) = \pr_\mk{p}(\mk{w}^\perp)$ contains a singular vector of type $B$, whereas $N_o(H_{2,1}^\Uplambda \ccdot o) = \mk{p}_{\upalpha_1}$ consists entirely of singular vectors of type $A$. Therefore, there cannot exist an orbit equivalence between these two actions. Now take the action of $H_{2,0}^\Uplambda$. The normal space of its singular orbit at $o$ is $T_o B_2 = \Rl H_{\upalpha_1} \oplus \mk{p}_{\upalpha_1}$, which is totally real. Again, the only action in (7) the normal spaces of whose singular orbit are totally real of dimension 3 is the action of $H_{1,(\pi/2,3)}^\Uplambda$ ($n \geqslant 2$). Regardless of the value of $n$, $\dim(\mk{w}^\perp \cap \mk{g}_{\upalpha_2}) \geqslant 2$, so $N_o(H_{1,(\pi/2,3)}^\Uplambda \ccdot o) = \pr_\mk{p}(\mk{w}^\perp)$ contains a 2-dimensional subspace $L$ of singular vectors of type $B$, while $N_o(H_{2,0}^\Uplambda \ccdot o)$ contains $\mk{p}_{\upalpha_1}$, a 2-dimensional subspace of singular vectors of type $A$. If there was an orbit equivalence between these two actions, there would be one fixing $o$, which would have to send $L$ onto a plane in $N_o(H_{2,0}^\Uplambda \ccdot o)$ overlapping with $\mk{p}_{\upalpha_1}$, which is impossible. Consequently, no action in (6) is orbit equivalent to an action in (7).

Finally, we proceed to the nilpotent construction method. To distinguish between the Euclidean and Hermitian inner products on $\mk{a} \oplus \mk{n}$, we will be adding the letter $H$ when talking about the latter. For example, $v \perp_H w$ means orthogonality with respect to the Hermitian inner product.

\textit{Nilpotent construction with} $j=2$. In this case we have:
\begin{align*}
    \mk{n}_2^1 &= \mk{g}_{\upalpha_2} \oplus \mk{g}_{\upalpha_1 + \upalpha_2} \simeq \Cx^n \oplus \Cx^n \cong \Cx^n \otimes \Cx^2, \\
    \mk{m}_2 &= \mk{z}_2 \oplus \mk{g}_2 = \mk{u}(n) \oplus \mk{so}(3,1) \simeq \mk{u}(n) \oplus \mk{sl}(2,\Cx), \\
    \mk{k}_2 &= \mk{u}(n) \oplus \mk{su}(2).
\end{align*}
It is not hard to show that the representation of $\mk{m}_2$ on $\mk{n}_2^1$ is equivalent to the external tensor product representation of $\mk{u}(n) \oplus \mk{sl}(2,\Cx)$ on $\Cx^n \otimes \Cx^2$. Consequently, the representation of $K_2^0 \simeq \U(n) \times \SU(2) \simeq \Ss(\U(n)\U(2))$ on $\mk{n}_2^1$ is equivalent to the isotropy representation of $\mathrm{Gr}^*(2, \Cx^{n+2}) = \SU(n,2)/\Ss(\U(n)\U(2))$. Pick some orthonormal bases (with respect to the Hermitian inner product) $e_1, \ldots, e_n \in \Cx^n$ and $f_1, f_2 \in \Cx^2$. Observe that when we think of $K_2^0 \curvearrowright \mk{n}_2^1$ as the isotropy representation of $\mathrm{Gr}^*(2, \Cx^{n+2})$, the two-dimensional subspace $\Rl (e_1 \otimes f_1) \oplus \Rl (e_2 \otimes f_2)$ is a maximal flat. This observation will prove useful later on. Now, let $\mk{v} \subseteq \mk{n}_2^1$ be an admissible and protohomogeneous subspace and let $v = v_1 + v_2 \in \mk{v}, \hspace{1pt} v_1 \in \mk{g}_{\upalpha_2}, \hspace{1pt} v_2 \in \mk{g}_{\upalpha_1 + \upalpha_2}$. First, mimicking the proof of Theorem 8 in \cite{berdntdominguez-vazquez}, we prove the following

\begin{lemma}\label{lemmaT}
Let $T = \left[ \begin{smallmatrix} 0 & -1 \\ 1 & 0 \end{smallmatrix} \right] \in \mk{so}(2) \subseteq \mk{su}(2) \subseteq \mk{m}_2$. If $[T, v_1] \perp_H v_2$, then either $v_1 = 0$ or $v_2 = 0$.
\end{lemma}
\vspace{-0.5em}
\begin{proof}[Proof of the lemma.]
Under the assumption $[T, v_1] \perp_H v_2$, we do not lose generality by taking $v_1 = r e_1 \otimes f_1$ and $v_2 = s e_2 \otimes f_2$ for some $r,s \in \Rl$. By protohomogeneity of $\mk{v}$, we have:
$$
\mk{v} = \Rl (r e_1 \otimes f_1 + s e_2 \otimes f_2) \oplus^\perp N_{\mk{k}_2}(\mk{v})(r e_1 \otimes f_1 + s e_2 \otimes f_2).
$$
Note that our assumption implies that the second summand here is actually orthogonal to both $e_1 \otimes f_1$ and $e_2 \otimes f_2$ (with respect to the Euclidean inner product). Now let $S + A \in N_{\mk{m}_2}(\mk{v}), \hspace{1pt} S \in \mk{u}(n), \hspace{1pt} A = \left[ \begin{smallmatrix} x & \hspace{0.55em}y \\ z & -x \end{smallmatrix} \right] \in \mk{sl}(2, \Cx)$. We compute:
\begin{align*}
    (S + A)(r e_1 \otimes f_1 &+ s e_2 \otimes f_2) = \\
    &= rS e_1 \otimes f_1 + sS e_2 \otimes f_2 + r e_1 \otimes (x f_1 + z f_2) + s e_2 \otimes (y f_1 - x f_2) \\
    &= r \Re(x) e_1 \otimes f_1 - s \Re(x) e_2 \otimes f_2 + (\text{terms in} \; (r e_1 \otimes f_1 + s e_2 \otimes f_2)^\perp).
\end{align*}
We see that $r \Re(x) e_1 \otimes f_1 - s \Re(x) e_2 \otimes f_2$ must lie in $\Rl (r e_1 \otimes f_1 + s e_2 \otimes f_2)$, which is only possible when either $r$ or $s$ is zero.
\end{proof}

Since $\Rl (e_1 \otimes f_1) \oplus \Rl (e_2 \otimes f_2)$ is a maximal flat, it intersects the isotropy orbit of each vector, i.e. there exists some $k \in K_2^0$ such that $\Ad(k)v \in \Rl (e_1 \otimes f_1) \oplus \Rl (e_2 \otimes f_2)$. It then follows from Lemma \ref{lemmaT} that $\Ad(k)v$ is proportional to either $e_1 \otimes f_1$ or $e_2 \otimes f_2$. Applying $T \in K_2^0$ if needed, we may assume the former is the case, so we simply assume $e_1 \otimes f_1 \in \mk{v}$.

Since $\mk{v}$ is protohomogeneous, we have $\mk{v} \subseteq \vspan\set{K_2^0 \ccdot (e_1 \otimes f_1)} = \mk{g}_{\upalpha_2} \oplus \Cx (e_1 \otimes f_2)$. We decompose $\mk{su}(2)$ as a vector space into two pieces:
$$
\ell = \left\{ \left. \begin{bmatrix} ia & \hspace{0.5em} 0 \\ 0 & -ia  \end{bmatrix} \right| a \in \Rl \right\}, \quad \ell^\perp = \left\{ \left. \begin{bmatrix} 0 & -\overline{z} \\ z & \hspace{0.5em} 0  \end{bmatrix} \right| z \in \Cx \right\},
$$
so $\mk{su}(2) = \ell \oplus \ell^\perp$. Clearly, $\mk{v} \subseteq \mk{g}_{\upalpha_2} \Leftrightarrow N_{\mk{k}_2}(\mk{v}) \subseteq \mk{u}(n) \oplus \ell$. One can easily see that every subspace $\mk{v}$ of $\mk{g}_{\upalpha_2}$ is automatically admissible. Now, we may regard such $\mk{v}$ as lying in the isometry Lie algebra $\mk{g}_1 = \widetilde{\mk{g}}_1$ of $B_1 \simeq \Cx H^{n+1}$, and $\mk{v}$ is protohomogeneous in $\mk{n}_2^1 \curvearrowleft M_2$ if and only if it is such in $\mk{g}_{\upalpha_2} \curvearrowleft \U(n)$ (which just means that it has constant K\"{a}hler angle). Moreover, the canonical extension of the resulting cohomogeneity-one action on $B_1$ has the same orbits as the action obtained from $\mk{v}$ by nilpotent construction performed on $M$. So such $\mk{v}$'s produce no new actions and we may assume $\mk{v}$ does not lie in $\mk{g}_{\upalpha_2}$ and thus $N_{\mk{k}_2}(\mk{v})$ has a nonzero projection in $\ell^\perp$. Let $S + A \in N_{\mk{k}_2}(\mk{v}), \hspace{1pt} S = (s_{ij})_{i,j=1}^n, \hspace{1pt} A = \left[ \begin{smallmatrix} ia & -\overline{z} \\ z & -ia \end{smallmatrix} \right], \hspace{1pt} z \ne 0$. We compute: \vspace{-0.4em}
\begin{align*}
    (S+A)(e_1 \otimes f_1) &= (ia + s_{11})e_1 \otimes f_1 + \sum_{i=2}^n s_{i1}e_i \otimes f_1 + z e_1 \otimes f_2, \\ \vspace{-0.5em}
    (S+A)^2(e_1 \otimes f_1) &= (\text{terms in} \hspace{0.4em} \mk{g}_{\upalpha_2} \oplus \Cx (e_1 \otimes f_2)) + \sum_{i=2}^n 2zs_{i1} e_i \otimes f_2.
\end{align*}
In order for this to lie in $\mk{v} \subseteq \mk{g}_{\upalpha_2} \oplus \Cx (e_1 \otimes f_2)$, we must have $s_{i1} = 0$ for $2 \leqslant i \leqslant n$, which means that $(ia + s_{11})e_1 \otimes f_1 + ze_1 \otimes f_2 \in \mk{v}$. But this forces $N_{\mk{k}_2}(\mk{v})$ to lie inside $\mk{u}(1) \oplus \mk{su}(2) \subseteq \mk{u}(n) \oplus \mk{su}(2)$, for otherwise we would have some $S + A \in N_{\mk{k}_2}(\mk{v}), S \in \mk{u}(n) \mysetminus \mk{u}(1),$ moving $(ia + s_{11})e_1 \otimes f_1 + ze_1 \otimes f_2$ out of $\mk{g}_{\upalpha_2} \oplus \Cx (e_1 \otimes f_2)$. The upshot of all this is that $\mk{v} \subseteq \Cx(e_1 \otimes f_1) \oplus \Cx(e_1 \otimes f_2)$ and $N_{\mk{m}_2}(\mk{v}) \subseteq \mk{u}(1) \oplus \mk{sl}(2, \Cx)$ (and we still assume $e_1 \otimes f_1 \in \mk{v}$). So we have reduced our problem to looking for admissible and protohomogeneous subspaces of $\Cx^2$ with respect to the tautological representation of $\mk{gl}(2, \Cx)$. Protohomogeneity singles out precisely subspaces of constant K\"{a}hler angle $\upvarphi$. We consider three cases:

\textit{Case 1: $\upvarphi = 0$}. Since we assume $\mk{v} \nsubseteq \mk{g}_{\upalpha_2}$, we must have $\mk{v} = \Cx(e_1 \otimes f_1) \oplus \Cx(e_1 \otimes f_2)$. But then we get a cohomogeneity-one action with a totally geodesic singular orbit isometric to $\mathrm{Gr}^*(2, \Cx^{n+3})$, see (3).

\textit{Case 2: $\upvarphi = \frac{\uppi}{2}$}. Without loss of generality, we pick $\mk{v} = \Rl(e_1 \otimes f_1) \oplus \Rl(e_1 \otimes f_2)$. In this case, $N_{\mk{gl}(2,\Cx)}(\mk{v}) = \mk{gl}(2, \Rl)$, so $\mk{v}$ fails to be admissible.

\textit{Case 3: $\upvarphi \in (0, \frac{\uppi}{2})$}. According to \cite[Proposition 7]{berndt_bruck}, we can take $\mk{v}$ to be the span of $(1,0)$ and $(i\cos \upvarphi, i\sin \upvarphi)$. If we write the coordinates on $\Cx^2$ as $z_1$ and $z_2$, then $\mk{v}$ is cut out by the equations $\Re (z_2) = 0$ and $\Im (z_2) = \Im (z_1) \tan\upvarphi$. We need to check whether the projection of $N_{\mk{sl}(2,\Cx)}(\mk{v})$ to the space of Hermitian traceless matrices is onto. A matrix $\left[ \begin{smallmatrix} x & \hspace{0.55em}y \\ z & -x \end{smallmatrix} \right] \in \mk{sl}(2, \Cx)$ normalizing $\mk{v}$ is subject to the following equations:
$$
\begin{cases}
\Re (z) = 0, \\
\Im (z) = \Im (x) \tan\upvarphi, \\
-\Im (z) \cos \upvarphi + \Im (x) \sin\upvarphi = 0, \\
\Re (z) \cos\upvarphi - \Re (x) \sin\upvarphi = (\Re (x) \cos\upvarphi + \Re (y) \sin\upvarphi) \tan\upvarphi.
\end{cases}
$$
The second and third equations are the same, so, simplifying, we are left with:
$$
\begin{cases}
\Re (z) = 0, \\
\Im (z) = \Im (x) \tan\upvarphi, \\
\Re (x) = - \Re (y) \dfrac{\tan \upvarphi}{2}.
\end{cases}
$$
These equations cut out a 3-dimensional subspace inside $\mk{sl}(2, \Cx)$. Its projection to the space of Hermitian traceless matrices is onto precisely when it does not intersect $\mk{su}(2)$. But these equations have a nontrivial solution in $\mk{su}(2)$, namely $\left[ \begin{smallmatrix} i & -i\tan \upvarphi \\ i\tan \upvarphi & -i \end{smallmatrix} \right]$, so $\mk{v}$ is not admissible.

\textit{Nilpotent construction with} $j=1$. In this case we have:
\begin{align*}
    \mk{n}_1^1 &= \mk{g}_{\upalpha_1} \oplus \mk{g}_{\upalpha_1 + \upalpha_2} \oplus \mk{g}_{\upalpha_1 + 2\upalpha_2} \simeq \Cx^{n+2}, \\
    \mk{m}_1 &= \mk{g}_1 \oplus \mk{z}_1 = \mk{su}(n+1,1) \oplus \mk{u}(1) \simeq \mk{u}(n+1,1), \\
    \mk{k}_1 &= \mk{s}(\mk{u}(n+1) \oplus \mk{u}(1)) \oplus \mk{u}(1) \simeq \mk{u}(n+1) \oplus \mk{u}(1).
\end{align*}
The representation of $\mk{m}_1$ on $\mk{n}_1^1$ can be shown to be equivalent to the tautological representation of $\mk{u}(n+1,1)$. We write $\Cx^{n+2} = \Cx^{n+1} \oplus \Cx$ and notice that $\mk{k}_1$ preserves the two summands. Consequently, a protohomogeneous subspace $\mk{v} \subseteq \Cx^{n+2}$ intersecting nontrivially with $\Cx^{n+1}$ (which is always the case when $\dim \mk{v} \geqslant 3$) must lie there entirely, in which case admissibility cannot be achieved. So we only need to consider 2-dimensional subspaces transversal to $\Cx^{n+1}$. Take any such $\mk{v}$. We can apply the same argument to its intersection with $\Cx \subseteq \Cx^{n+2}$, so it must be trivial as well. In other words, the projection $\pr_{\Cx^{n+1}}(\mk{v})$ is two-dimensional. Observe that this projection is protohomogeneous in $\Cx^{n+1}$ with respect to $\U(n+1)$ and thus has constant K\"{a}hler angle $\upvarphi$. Again, according to \cite[Proposition 7]{berndt_bruck}, we may assume 
$$
\mk{v} = \langle v_1, v_2 \rangle_\Rl = \langle e_1 + ae_{n+2}, i\cos \upvarphi e_1 + i\sin \upvarphi e_2 + be_{n+2} \rangle_\Rl, \; a,b \ne 0.
$$
Moreover, acting by $\U(1)$ in the last coordinate, we can take $a > 0$. Finally, since $K_1^0$ acts transitively on the unit circle in $\mk{v}$, we must have $|b| = a$. Consequently, we can confine ourselves to working in $\Cx e_1 \oplus \Cx e_2 \oplus \Cx e_{n+2}$. Let us denote the corresponding coordinates by $x,y,$ and $z$, respectively. Note that $\Im (b) \ne 0$, for otherwise we would have a nontrivial intersection $\mk{v} \cap \Cx^{n+1}$. We see that $\mk{v}$ is cut out by the equations
\begin{equation}\label{nastysystem}
    \begin{cases}
        \Re (y) = 0, \\
        \Im (x) \sin \upvarphi = \Im (y) \cos \upvarphi, \\
        \Im (b) \Im (x) = \Im (z) \cos \upvarphi, \\
        a \Re (x) = \Re (z) - \dfrac{\Re (c)}{\Im (c)}\Im (z).
    \end{cases}
\end{equation}
The subspace $\mk{v}$ is admissible if and only if for each $w \in \Cx^{n+1}$ there exists
$$
S = \left[ \begin{array}{c | c}
            \begin{array}{c c c}
                & & \\
                & A & \\
                &  &
            \end{array} & w \\
            \hline
            w^* & c
        \end{array} \right] \in \mk{u}(n+1,1)
$$
preserving $\mk{v}$. The fact that $S$ lies in $\mk{u}(n+1,1)$ means that $A \in \mk{u}(n+1)$ and $\Re (c) = 0$. Applying $S$ to $(1,0,\ldots,a)$ and $(i\cos\upvarphi,i\sin\upvarphi,0,\ldots,b)$, we get vectors 
$$
\left[ \begin{gathered} a_{11} + aw_1 \\
a_{21} + aw_2 \\ \cdots \\
\overline{w}_1 + iac \end{gathered} \right] \quad \text{and} \quad \left[ \begin{gathered} ia_{11}\cos\upvarphi + ia_{12}\sin\upvarphi + w_1 b \\
ia_{21}\cos\upvarphi + ia_{22}\sin\upvarphi + w_2 b \\ \cdots \\
i\overline{w}_1 \cos\upvarphi + i\overline{w}_2 \sin\upvarphi + ibc \end{gathered} \right].
$$
Note that $a_{12} = \overline{a}_{21}$ and $\Re (a_{11}) = \Re (a_{22}) = 0$. In order to lie in $\mk{v}$, both $S(v_1)$ and $S(v_2)$ must satisfy equations (\refeq{nastysystem}) (and all their other coordinates should be 0, but we do not care about that), which altogether gives a system of 8 equations, and we are left with a rather daunting and tedious task of solving it. We are not going to write all these equations here; instead, let us label them (i) through (viii), where equations (i) to (iv) are for $S(v_1)$ substituted into (\refeq{nastysystem}) and (v) to (viii) are for $S(v_2)$ substituted there. Fortunately, there is no need to solve the entire system. Indeed, provided that $\upvarphi \ne 0$, use (i) to express $\Re (a_{21})$ in terms of $\Re (w_2)$ and plug the result into (vi) to get a linear dependency on the real and imaginary parts of $w_1$ and $w_2$. Therefore, such $\mk{v}$ cannot be admissible.

Now let $\upvarphi = 0$, which just means that $\pr_{\Cx^{n+1}}(\mk{v})$ is a complex line. The first two equations in (\refeq{nastysystem}) then simply mean that $y = 0$. Assume first that $\Re (b) \ne 0$. Use equations (iii) and (iv) to express $a_{11}$ and $c$ in terms of $w_1$. Equation (vii) will tell us that $a \Im (b) = 1$. Inserting it all into (viii) yields a linear dependency on $\Re (w_1)$ and $\Im (w_1)$, which implies that coefficients in this dependency must be zero. But one can easily compute that the coefficient of $\Im (w_1)$ is $(a-\frac{1}{a})^2 + \Re (b)^2$, which cannot be zero. So we end up with the case $\Re (b) = 0$. Equation (iv) yields $a = 1$ and thus $b = i$, so $\mk{v}$ becomes the complex line spanned inside $\Cx^{n+2}$ by $e_1 + e_{n+2}$. It is not hard to verify that this subspace is admissible and protohomogeneous. Now we want to see what this $\mk{v}$ looks like with respect to the restricted root space decomposition. To this end, we need to have an explicit identification between $\mk{m}_1$ and $\mk{u}(n+1,1)$ and also between $\mk{n}_1^1$ and $\Cx^{n+2}$. We have:
$$
\mk{m}_1 = \left\{ \left. \left[ \begin{array}{c | c}
            \begin{array}{c c c}
                & & \\
                & A & \\
                &  &
            \end{array} & \begin{array}{c c c}
                & & \\
               0 & 0 & u \\
                & &
            \end{array} \\
            \hline
            \begin{array}{c c c}
                & 0 & \\
                & 0 & \\
                & u^* &
            \end{array} & \begin{array}{c c c}
                ib & & \\
                & ib & \\
                & & ia
            \end{array}
        \end{array} \right] \right| A \in \mk{u}(n+1), u \in \Cx^{n+1}, \hspace{1.5pt} a,b \in \Rl, \hspace{1pt} \tr = 0 \right\},
$$
where $u$ and $u^*$ are regarded as column and row vectors, respectively. Similarly,
$$
\mk{n}_1^1 = \left\{ \left. \left[ \begin{array}{c | c}
            \begin{array}{c c c}
                & & \\
                & 0 & \\
                &  &
            \end{array} & \begin{array}{c c c}
                & & \\
               -v & v & 0 \\
                & &
            \end{array} \\
            \hline
            \begin{array}{c c c}
                & v^* & \\
                & v^* & \\
                & 0 &
            \end{array} & \begin{array}{c c c}
                0 & 0 & -\overline{w} \\
                0 & 0 & -\overline{w} \\
                -w & w & 0
            \end{array}
        \end{array} \right] \right| v \in \Cx^{n+1}, w \in \Cx \right\},
$$
where the top left zero block is of the size $(n+1) \times (n+1)$ and $\pm v$ and $v^*$ are again regarded as column and row vectors, respectively. Now, sending the matrix in the definition of $\mk{m}_1$ to
$$
\left[ \begin{array}{c | c}
            \begin{array}{c c c}
                & & \\
                & A-ibE & \\
                &  &
            \end{array} & u \\
            \hline
            u^* & i(a+b)
        \end{array} \right]
$$
gives a Lie algebra isomorphism $\mk{m}_1 \isoto \mk{u}(n+1,1)$. Similarly, sending the matrix in the definition of $\mk{n}_1^1$ to $(v,w)$ (here $v$ is a row vector) gives an $\Rl$-linear isomorphism $\mk{n}_1^1 \isoto \Cx^{n+2}$. Moreover, under these isomorphisms, the adjoint representation of $\mk{m}_1$ on $\mk{n}_1^1$ becomes the tautological representation of $\mk{u}(n+1,1)$. Note, however, that the isomorphism $\mk{n}_1^1 \isoto \Cx^{n+2}$ is not $\Cx$-linear with respect to the complex structure on $\mk{n}_1^1$ induced from $\mk{a} \oplus \mk{n}$! It is $\Cx$-linear in $v$ but $\Cx$-antilinear in $w$. This subtlety is important if we want to transfer our subspace $\mk{v}$ from $\Cx^{n+2}$ to $\mk{n}_1^1$: while it is a complex line in $\Cx^{n+2}$, it will become a totally real 2-dimensional subspace in $\mk{n}_1^1$. In fact, if we modify $\mk{v}$ slightly by an element from $\U(n+1,1)$ to make it a complex line spanned by $(0,\ldots,0,1,1)$, then the subspace of $\mk{n}_1^1$ it corresponds to is precisely $\mk{g}_{\upalpha_1}$, which can be easily seen from our description of the isomorphism above. We thus have a cohomogeneity-one action given by the Lie subgroup $H_{1,\mk{g}_{\upalpha_1}}$, whose Lie algebra is
\begin{align*}
    \mk{h}_{1, \mk{g}_{\upalpha_1}} &= N_{\mk{m}_1}(\mk{n_1} \ominus \mk{g}_{\upalpha_1}) \oplus \mk{a}^1 \oplus (\mk{n_1} \ominus \mk{g}_{\upalpha_1}) \\
    &= (\mk{k}_0 \oplus \mk{a}^1 \oplus \mk{g}_{\upalpha_2} \oplus \mk{g}_{2\upalpha_2}) \oplus \mk{a}^1 \oplus (\mk{n_1} \ominus \mk{g}_{\upalpha_1}) \\
    &= \mk{k}_0 \oplus \mk{a} \oplus \mk{n}_2 = \mk{h}_{2,1}^\Uplambda,
\end{align*}
so $H_{1,\mk{g}_{\upalpha_1}} = H_{2,1}^\Uplambda$, and this action was taken into account in (6). Altogether, we conclude that the nilpotent construction does not yield any new actions for $M$.
\end{proof}

\printbibliography

@article{berndttamarufoliations,
    author = {J. Berndt and H. Tamaru},
    title = {Homogeneous codimension one foliation on noncompact symmetric spaces},
    year = {2003},
    journal = {J. Differential Geom.},
    volume = {63},
    number = {1},
    pages = {1-40}
}

@article{berndttamarusingtotgeod,
    author = {J. Berndt and H. Tamaru},
    title = {Cohomogeneity one actions on noncompact symmetric spaces with a totally geodesic singular orbit},
    year = {2004},
    journal = {Tohoku Math J.},
    volume = {56},
    number = {2},
    pages = {163-177}
}

@article{berndttamarucohomogeneityone,
    author = {J. Berndt and H. Tamaru},
    title = {Cohomogeneity one actions on symmetric spaces of noncompact type},
    year = {2013},
    journal = {J. reine angew. Math.},
    volume = {683},
    pages = {129-159}
}

@article{hyperpolarfoliations,
    author = {J. Berndt and J. C. D\'{i}az-Ramos and H. Tamaru},
    title = {Hyperpolar homogeneous foliations on symmetric spaces of noncompact type},
    year = {2010},
    journal = {Journal of Differential Geometry},
    volume = {86},
    number = {2},
    pages = {191-236}
}

@book{helgason,
    author = {S. Helgason},
    title = {Differential geometry, Lie groups, and symmetric spaces},
    year = {1978},
    publisher = {Academic Press},
}

@book{knapp,
    author = {A. Knapp},
    title = {Lie Groups Beyond an Introduction},
    year = {2002},
    edition = {2nd edn},
    publisher = {Birkh\"{a}user},
    series = {Progress in Mathematics},
    volume = {140}
}

@book{submanifoldsholonomy,
    title = {Submanifolds and Holonomy},
    author = {J. Berndt and S. Console and C. E. Olmos},
    year = {2016},
    edition = {2nd edn},
    publisher = {CRC Press}
}

@article{mostowsubgroups,
    title = {The extensibility of local {Lie} groups of transformations and groups on surfaces},
    author = {G. D. Mostow},
    year = {1950},
    journal = {Annals of Mathematics},
    volume = {52},
    number = {3},
    pages = {606-636}
}

@article{berdntdominguez-vazquez,
    title = {Cohomogeneity one actions on some noncompact symmetric spaces of rank two},
    author = {J. Berndt and M. Dom\'{i}nguez-V\'{a}zquez},
    year = {2015},
    journal = {Transformation groups},
    volume = {20},
    pages = {921-938}
}

@misc{protohomogeneous,
      title={Homogeneous and inhomogeneous isoparametric hypersurfaces in rank one symmetric spaces}, 
      author={Jos\'{e} Carlos D\'{i}az-Ramos and Miguel Dom\'{i}nguez-V\'{a}zquez and Alberto Rodr\'{i}guez-V\'{a}zquez},
      year={2020},
      eprint={2005.09314},
      archivePrefix={arXiv},
      primaryClass={math.DG}
}

@article{berndt_bruck,
    title = {Cohomogeneity one actions on hyperbolic spaces},
    author = {J. Berndt and M. Br\"{u}ck},
    year = {2001},
    journal = {J. Reine Angew. Math.},
    volume = {541},
    pages = {209-235}
}

@article{berndttamarurankone,
    author = {J. Berndt and H. Tamaru},
    title = {Cohomogeneity one actions on noncompact symmetric spaces of rank one},
    year = {2007},
    journal = {Transactions of the American Mathematical Society},
    volume = {359},
    pages = {3425–3438}
}

@article{berndtgrassmannian,
    author = {J. Berndt},
    title = {Riemannian geometry of complex two-plane {Grassmannians}},
    year = {1997},
    journal = {Rend. Sem. Mat. Univ. Politec. Torino},
    volume = {55},
    pages = {19-83}
}

@article{kollrossclassification,
    author = {A. Kollross},
    title = {A classification of hyperpolar and cohomogeneity one actions},
    year = {2002},
    journal = {Transactions of the American Mathematical Society},
    volume = {354},
    number = {2},
    pages = {571-612}
}

@book{borel_ji,
    author = {A. Borel and L. Ji},
    title = {Compactifications of Symmetric and Locally Symmetric Spaces},
    year = {2006},
    publisher = {Birkh\"{a}user Basel}
}

@article{wolf_complexforms,
    author = {J. Wolf},
    title = {Complex forms of quaternionic symmetric spaces},
    year = {2005},
    publisher = {Birkh\"{a}user Boston},
    pages = {265-277},
    volume = {234},
    series = {Progress in Mathematics},
    booktitle = {Complex, Contact and Symmetric Manifolds}
}

@article{jaffee_realforms,
    author = {H. Jaffee},
    title = {Real forms of {Hermitian} symmetric spaces},
    year = {1975},
    journal = {Bull. Amer. Math. Soc.},
    volume = {81},
    number = {2},
}

@online{IIformula,
    author = {Ivan Solonenko},
    title = {A convenient formula for the second fundamental form in symmetric spaces},
    url = {https://drive.google.com/file/d/1hwfuZ-U3e2uCFhfix_f18snNs6UAaY4T/view?usp=sharing},
    year = {2021}
}

@misc{solonenko2021homogeneous,
      title={Homogeneous codimension-one foliations on reducible symmetric spaces of noncompact type}, 
      author={Ivan Solonenko},
      year={2021},
      eprint={2112.02189},
      archivePrefix={arXiv},
      primaryClass={math.DG}
}

@misc{solonenko2022automorphisms,
      title={Automorphisms of real semisimple {Lie} algebras and their restricted root systems}, 
      author={Ivan Solonenko},
      year={2022},
      eprint={2208.01131},
      archivePrefix={arXiv},
      primaryClass={math.DG}
}

\end{document}